\documentclass[12pt]{article}
\usepackage{graphicx}
\usepackage{amsmath}
\usepackage{amsthm}
\usepackage{amssymb}
\usepackage{amsfonts}
\usepackage{mathrsfs,bbm}
\usepackage{verbatim}
\usepackage{mathrsfs}
\usepackage[all]{xy}
\usepackage{color}
\usepackage{pdfcolmk}

\numberwithin{equation}{section}

\theoremstyle{plain}
\newtheorem{theorem}{Theorem}
\newtheorem{corollary}{Corollary}

\newtheorem{lemma}{Lemma}

\newtheorem{proposition}{Proposition}

\theoremstyle{definition}
\newtheorem{definition}{Definition}

\theoremstyle{remark}
\newtheorem{example}{Example}
\newtheorem{remark}{Remark}

\renewcommand{\phi}{\varphi}
\newcommand{\Hc}{\mathcal{H}}

\newcommand{\Oc}{\mathcal{O}}

\newcommand{\Bc}{\mathcal{B}}
\newcommand{\Mc}{\mathcal{M}}

\newcommand{\F}{\mathbb{F}}

\newcommand{\Ac}{\mathcal{A}}

\newcommand{\A}{\mathbb{A}}

\newcommand{\C}{\mathbb{C}}
\newcommand{\R}{\mathbb{R}}
\newcommand{\Q}{{\mathbb{Q}}}

\newcommand{\Z}{\mathbb{Z}}

\newcommand{\N}{\mathbb{N}}
\newcommand{\pfk}{\mathfrak{p}}

\newcommand{\Spec}{\mathrm{Spec}}

\newcommand{\Spev}{\mathrm{Spev}}
\newcommand{\Speh}{\mathrm{Speh}}

\newcommand{\Spa}{\mathrm{Spa}}

\newcommand{\Halos}{\mathfrak{Halos}}

\newcommand{\Ker}{\mathrm{Ker}}

\newcommand{\Hol}{\mathcal{H}ol}

\newcommand{\Frac}{\mathrm{Frac}}

\newcommand{\Top}{\textsc{(Top)}}
\newcommand{\SNRings}{\textsc{(SNRings)}}
\newcommand{\Rings}{\textsc{(Rings)}}

\newcommand{\limind}{{\underset{\longrightarrow}{\mathrm{lim}}\,}}

\newcommand{\lextimes}{{\scriptstyle{[\times]}}}

\title{Global analytic geometry}
\author{Fr\'ed\'eric Paugam
\footnote{Universit\'e Paris 6,
Institut de Math\'ematiques de Jussieu.}}

\begin{document}
\maketitle



\begin{abstract}
We define, formalizing Ostrowski's classification of seminorms on $\Z$,
a new type of valuation of a ring that combines the notion
of Krull valuation with that of a multiplicative seminorm.
This definition partially restores the broken symmetry
between archimedean and non-archimedean valuations.
This also allows us to define a notion of global analytic
space that reconciles Berkovich's notion of analytic space
of a (Banach) ring with Huber's notion of non-archimedean
analytic spaces. After defining natural generalized valuation spectra and
computing the spectrum of $\Z$ and $\Z[X]$,
we define analytic spectra and sheaves of analytic functions
on them.
\end{abstract}

\newpage
\tableofcontents
\newpage


\section*{Introduction}
Many interesting results on polynomial equations can be proved
using the mysterious interactions
between algebraic, complex analytic and p-adic analytic geometry.
The aim of global analytic geometry is to construct a category of spaces which contains
these three geometries.

Remark that the study of a given polynomial equation $P(X,Y)=0$ is completely
equivalent to the study of the corresponding commutative ring
$A=\Z[X,Y]/(P(X,Y))$. To associate a geometry to a given ring $A$, one
first needs to define what the points, usually
called \emph{places} of this geometry are.
There are many different definitions of what a place of a ring is.
K\"urch\'ak (1912) and Ostrowski (1917) use real valued multiplicative (semi)norms,
Krull (1932) uses valuations with values in abstract totally ordered
groups and Grothendieck (1958) uses morphisms to fields.
There is a natural geometry associated to each type of places:
\begin{enumerate}
\item the theory of schemes (see \cite{Gro1}) ensues from Grothendieck's viewpoint,
\item Berkovich's geometry (see \cite{Berkovich1}) ensues from Ostrowski's viewpoint,
\item Zariski/Huber's geometry (see \cite{E-Artin1} and \cite{Huber1})
ensues from Krull's viewpoint.
\end{enumerate}
For some number theoretical purposes like the study of functional
equations of L-functions, a dense part of the mathematical community
tend to say that one should try to
\begin{center}
``have a perfect analogy between archimedean and non-archimedean valuations''.
\end{center}

As an illustration of this problematic, one can recall that the functional
equation
$$\hat{\zeta}(s)=\hat{\zeta}(1-s)$$
of Riemann's completed zeta function
$$\hat{\zeta}(s)=\pi^{-s/2}\Gamma(s/2)\prod_{p}\frac{1}{1-p^{-s}}$$
cannot be studied geometrically without handling the archimedean factor
$\zeta_{\infty}(s)=\pi^{-s/2}\Gamma(s/2)$
(that corresponds to the archimedean absolute value on $\Q$)
in the given geometrical setting. The question is even more interesting
for higher dimensional varieties over $\Z$ because the proof of the
functional equation of their zeta function is a widely open question.
Arakelov geometry (see \cite{Soule4} and \cite{Durov-2007}) studies archimedean properties
of algebraic varieties, and this results in
a deep improvement of our understanding of the geometry of numbers,
but in no proof of the functional equation. A good reason to think that
global analytic spaces are useful for this question is the following
definition due to Berkovich (private email) which is easily seen
to be equivalent to the definition of Tate's thesis \cite{Tate-These},
which is the corner stone of modern analytic number theory.

\begin{definition}
Let $|.|_{0}$ be the trivial seminorm on $\Z$ and
$U\subset \Mc(\Z)$ be the complement of it in the
analytic space of $\Z$. Let $\Oc$ be the sheaf of
analytic functions on $\Mc(\Z)$. The ad\`eles of
$\Z$ are the topological ring
$$\A:=(j_{*}\Oc_{U})_{(|.|_{0})}$$
of germs of analytic functions at the trivial seminorm.
\end{definition}

This geometric definition of ad\`eles opens the road to various higher
dimensional generalizations and shows that the topological sheaf of functions on an analytic space is a good replacement of ad\`eles in higher dimensions. It also
shows, once combined with the ideas already present in Emil Artin's book
\cite{E-Artin1}, that it is worth continuing to think about the following na\"ive but
fundamental question\footnote{Question which will not get a satisfying answer
in this paper.}: what is a number?

Another motivation for defining a natural setting for global analytic
geometry is that, in the conjectural correspondence between motives
and automorphic representations due to Langlands, a lot of
(non-algebraic) automorphic representations are left aside.
If one enlarges the category of motives by adding the cohomology
of natural coefficient systems on analytic varieties, one can
hope to obtain a full Langlands correspondence between
certain ``analytic motivic coefficients'' and all automorphic representations.
The definition of these analytic motivic coefficients is at this time not
at all clear and far beyond the scope of the present paper.

As a first step in the direction of this long term allusive objective, we define in
this text a simple notion of generalized valuation (with tempered growth)
that allows one to mix the main viewpoints
of places in a definition that contains but does not distinguish
archimedean and non-archimedean valuations.
This definition ensues a new setting of global analytic geometry,
that is probably not definitive, but has the merit to give one positive
and computable answer to the question:
``is it possible to treat all places on equality footing''.

The first construction in the direction of a global analytic geometry
is due to Berkovich \cite{Berkovich1}, chapter 1 (see also
Poineau's thesis \cite{Poineau}): he considers
spaces of multiplicative seminorms on commutative Banach rings, giving
the example of the Banach ring $(\Z,|.|_\infty)$ of integers with their
archimedean norm. He defines a category of global analytic spaces that
contains complex analytic and his non-archimedean analytic spaces. One of
the limitations of his construction is that a good theory of non-archimedean
coherent analytic sheaves sometimes imposes the introduction of a
Grothendieck topology (the rigid analytic topology defined by Tate \cite{Tate2})
on his analytic spaces, which is
essentially generated by affinoid domains $\{x|\,|a(x)|\leq|b(x)|\neq 0\}$.
It was proved by Huber \cite{Huber1} that in the non-archimedean case,
the topos of sheaves for this Grothendieck topology has enough points,
so that it corresponds to a usual topological space.
This space is the valuation spectrum of the corresponding adic ring, whose
points are bounded continuous Krull valuations.
The non-archimedean components of Berkovich's analytic spaces give
subspaces (or more precisely retractions) of Huber's valuation spectra
corresponding to rank one valuations.
However, there is no construction in the literature that combines
Huber's viewpoint (which is nicer from an abstract sheaf theoretic point of view)
with Berkovich's viewpoint (which has the advantage of giving separated spaces and allowing to naturally incorporate archimedean components).

We propose in this text a new kind of analytic spaces that gives a natural
answer to this simple open problem.
The construction is made in several steps. We start in Section 1 by studying
the category of halos, which is the simplest category that contains the category of rings,
and such that Krull valuations and seminorms are morphisms in it.
In Section 2, we define a new notion of tempered generalized valuation which entails
a new notion of place of a ring. In Section 3, we use this new notion of place to
define a topological space called the harmonious spectrum of a ring. In Section 4, we give a definition of the analytic spectrum and define analytic spaces using
local model similar to Berkovich's \cite{Berkovich1}, 1.5.
We finish by computing in detail the points of the global analytic affine line over $\Z$
and proposing another approach to the definition of analytic functions.


All rings and semirings of this paper will be unitary, commutative
and associative.

\section{Halos}
We want to define a category that contains rings fully faithfully
and such that valuations and (multiplicative) seminorms both
are morphisms in this category. The most simple way
to do this is to use the category whose objects are semirings
equipped with a partial order compatible to their operations
and whose morphisms are maps $f:A\to B$ such that
$f(0)=0$, $f(1)=1$, and that fulfill the subadditivity and submultiplicativity
conditions
$$
\begin{array}{c}
f(a+b)\leq f(a)+f(b),\\
f(ab)\leq f(a)f(b).
\end{array}
$$
An object of this category will be called a halo.
It is often supposed, for localization purposes, that $f$ is strictly multiplicative,
i.e., $f(ab)=f(a)f(b)$. We will see that this hypothesis is sometimes
too restrictive for our purposes.

\subsection{Definition and examples}
\begin{definition}
A halo is a semiring $A$ whose underlying set is equipped with
a partial order $\leq$ which is compatible to its operations:
$x\leq z$ and $y\leq t$ implies $xy\leq zt$ and $x+y\leq z+t$.
A morphism between two halos is an increasing map
$f:A\to B$ which is submultiplicative, i.e.,
\begin{itemize}
\item $f(1)=1$,
\item $f(ab)\leq f(a)f(b)$ for all $a,b\in A$,
\end{itemize}
and subadditive, i.e.,
\begin{itemize}
\item $f(0)=0$,
\item $f(a+b)\leq f(a)+f(b)$ for all $a,b\in A$.
\end{itemize}
The category of halos is denoted $\Halos$.
A halo morphism is called square-multiplicative (resp. power-multiplicative,
resp. multiplicative)
if $f(a^2)=f(a)^2$ (resp. $f(a^n)=f(a)^n$, resp. $f(ab)=f(a).f(b)$)
for all $a,b\in A$ and $n\in \N$.
The categories of halos with square-multiplicative (resp. power-multiplicative, resp.
multiplicative) morphisms between them is denoted
$\Halos^{sm}$ (resp. $\Halos^{pm}$, resp. $\Halos^{m}$). 
\end{definition}

Let $B$ be a semiring. The trivial order on $B$ gives
it a halo structure that we will denote $B_{triv}$.
If $A$ is a halo and $f:A\to B_{triv}$ is a halo morphism,
then $f$ is automatically a semiring morphism.
The functor $B\mapsto B_{triv}$ gives a fully faithful
embedding of the category of semirings into the
categories $\Halos$, $\Halos^{sm}$, $\Halos^{pm}$ and $\Halos^m$.

\begin{remark}
The field $\R$ equipped with its usual order is not a halo
because this order is not compatible with the multiplication
of negative elements. This shows that a halo is something
different of the usual notion of an ordered ring used
in the literature.
\end{remark}

We will now prove that rings have only one halo structure:
the trivial one.

\begin{lemma}
A halo which is a ring has necessarily a trivial order.
\end{lemma}
\begin{proof}
It is mainly the existence of an inverse for addition which implies that
the order is trivial. Suppose that $a\leq b\in A$. Since $-b-a=-b-a$ and
the sum respects the order,
we have $-b-a+a\leq -b-a+b$, i.e. $-b\leq -a$. We know that $-1=-1$ so
$(-1).(-b)\leq (-1).(-a)$. Now adding $b$ to
$0=(-1+1).(-b)=(-1).(-b)+(-b)$ implies $(-1).(-b)=b$.
So we have $b\leq a$, which implies that $b=a$.
\end{proof}

\begin{remark}
From now on, we will often identify a ring with its unique (trivial) halo structure.
\end{remark}

\begin{definition}
A halo whose underlying semiring is a semifield is called
an aura.
\end{definition}

\begin{definition}
A halo $A$ is called positive if $0<1$ in $A$.
\end{definition}

\begin{remark}
If a halo $A$ is positive, then
$0=0.a\leq 1.a=a$ for all $a\in A$.
\end{remark}

\begin{remark}
\label{smaller-than}
If a totally ordered aura $R$ is positive and $0<r\neq 1$ in $R$,
then there exists $0<r'<r$ in $R$. Indeed, if $r>1$, then
$r'=1/r<1<r$ and if $r<1$, then $r'=r^2<r$.
\end{remark}

We now give three interesting examples of halo morphisms.

\begin{example}
The semifield $\R_{+}$ equipped with its usual laws
and ordering is a totally ordered positive aura.
If $A$ is a ring, then a classical seminorm on $A$ is
exactly a halo morphism
$$|.|:A\to \R_{+}.$$
\end{example}

\begin{example}
More generally, if $R$ is a real closed field, the semifield $R_{\geq 0}=\{x^2|x\in R\}$
of its positive elements (i.e. its positive cone, that is also its squares since $R$
is real closed) equipped with its usual laws and ordering is a totally ordered positive aura.
If $A$ is a ring, we can thus generalize seminorms by using halo morphisms
$$|.|:A\to R_{\geq 0}.$$
These have the advantage to be tractable with model theoretic
methods because the theory of real closed fields admits elimination of
quantifiers (see \cite{Schwartz-valuation}).
\end{example}

\begin{example}
If $\Gamma$ is a totally ordered group (multiplicative notation), the semigroup
$R_{\Gamma}:=\{0\}\cup \Gamma$ equipped with the multiplication and order
such that $0.\gamma=0$ and $0\leq \gamma$ for all $\gamma\in \Gamma$
and with addition $a+b=\max(a,b)$ is also a totally ordered positive aura.
Its main difference with the positive cone of a real closed field is
that it does not have the simplification property
$$x+y\leq x+z\Rightarrow y\leq z.$$
If $A$ is a ring, a halo morphism
$$|.|:A\to R_{\Gamma}$$
is exactly a valuation in Krull's sense. If $\Gamma^{div}$ is the divisible
closure of $\Gamma$ and $R=\R((\Gamma^{div}))$ is the corresponding
real closed field, we can associate to $|.|$ a halo morphism
$$|.|:A\to R_{\geq 0}$$
given by composition with the natural multiplicative
halo morphism \mbox{$R_{\Gamma}\to R_{\geq 0}$}.
\end{example}

\subsection{Multiplicative elements and localization}
\begin{definition}
\label{definition-multiplicative}
Let $f:A\to B$ be a halo morphism. The set of multiplicative elements
for $f$ in $A$ is the set $M_{f}$ of $a\in A$ such that
\begin{enumerate}
\item $f(a)\in B^\times$,
\item for all $b\in A$, $f(ab)=f(a)f(b)$.
\end{enumerate}
\end{definition}

\begin{proposition}
\label{prop-localization-multiplicative}
Let $A$ be a halo and $S\subset A$ be a multiplicative
subset (i.e. a subset that contains $1$ and is stable by multiplication).
Then the localized semiring $A_{S}$ is equipped with a natural
halo structure such that if $f:A\to B$ is a halo morphism
with $S\subset M_{f}$ then $f$ factorizes uniquely through
the morphism $A\to A_{S}$.
\end{proposition}
\begin{proof}
The localized semiring $A_{S}$ is defined
as the quotient of the product $A\times S$ by the relation
$$(a,s)R(b,t)\Leftrightarrow \exists\,u\in S,\, uta=usb.$$
Recall that the sum and product on $A\times S$ are defined by
$(a,s)+(b,t)=(at+bs,st)$ and $(a,s).(b,t)=(ab,st)$.
We put on $A\times S$ the pre-order given by
$$(a_{1},s_{1})\leq (a_{2},s_{2})\Leftrightarrow \exists\, u\in S,\,ua_{1}s_{2}\leq us_{1}a_{2}.$$
Remark that the equivalence relation associated to this pre-order is exactly
the equivalence relation we want to quotient by.
This order is compatible with the two operations given
above. Indeed, if $(a_{1},s_{1})\leq (a_{2},s_{2})$ and
$(b_{1},t_{1})\leq (b_{2},t_{2})$, then by definition there exist $u,v\in S$
with
$ua_{1}s_{2}\leq us_{1}a_{2}$ and
$vb_{1}t_{2}\leq vt_{1}b_{2}$, so that
$uva_{1}b_{1}s_{2}t_{2}\leq uva_{2}b_{2}s_{1}t_{1}$
by compatibility of the order with multiplication in $A$.
This shows that 
$(a_{1},s_{1}).(b_{1},t_{1})\leq (a_{2},s_{2}).(b_{2},t_{2})$.
Now $(a_{1},s_{1})+(b_{1},t_{1})=(a_{1}t_{1}+b_{1}s_{1},s_{1}t_{1})$
and
$(a_{2},s_{2})+(b_{2},t_{2})=(a_{2}t_{2}+b_{2}s_{2},s_{2}t_{2})$.
Remark that
$(a_{1}t_{1}+b_{1}s_{1}).s_{2}t_{2}=a_{1}t_{1}s_{2}t_{2}+b_{1}s_{1}s_{2}t_{2}$
and
$(a_{2}t_{2}+b_{2}s_{2}).s_{1}t_{1}=a_{2}t_{2}s_{1}t_{1}+b_{2}s_{2}s_{1}t_{1}$.
The inequalities
$ua_{1}s_{2}\leq us_{1}a_{2}$ and
$vb_{1}t_{2}\leq vt_{1}b_{2}$
imply
$ua_{1}s_{2}.(t_{1}t_{2}v)\leq us_{1}a_{2}.(t_{1}t_{2}v)$
and
$vb_{1}t_{2}.(s_{1}s_{2}u)\leq vt_{1}b_{2}.(s_{1}s_{2}u)$.
Adding these inequalities and changing parenthesis, we get
$(uv).(a_{1}t_{1}+b_{1}s_{1}).(s_{2}t_{2})\leq 
(uv).(a_{2}t_{2}+b_{2}s_{2}).(s_{1}t_{1})$
which shows that
$(a_{1},s_{1})+(b_{1},t_{1})\leq (a_{2},s_{2})+(b_{2},t_{2})$,
as claimed. So the operations on $A\times S$ are compatible
with the defined pre-order. The quotient order of this pre-order
is exactly the underlying set of the localized semiring,
and is equipped with a canonical order compatible to its operations,
i.e. a canonical halo structure.
Let $f:A\to B$ be a halo morphism such that $S\subset M_{f}$.
Then we can define $\tilde{f}:A\times S\to B$ by
$\tilde{f}(a,s)=f(a)/f(s)$. This is well defined since $f(s)$ is invertible in $B$.
Suppose now that $(a_{1},s_{1})\leq (a_{2},s_{2})$ in $A\times S$, which means
that there exists $u\in S$ such that $ua_{1}s_{2}\leq us_{1}a_{2}$.
We then have $f(ua_{1}s_{2})\leq f(us_{1}a_{2})$ and since 
$S\subset M_{f}$, this gives
$f(a_{1})f(s_{2})\leq f(s_{1})f(a_{2})$ so that
$\tilde{f}(a_{1},s_{1})\leq \tilde{f}(a_{2},s_{2})$. This shows
that $\tilde{f}$ factorizes through $A_{S}$. Now it remains to
show that the obtained map, denoted $g$, is also subadditive
and submultiplicative. We already know that for all
$a,b\in A$, $f(a+b)\leq f(a)+f(b)$ and $f(ab)\leq f(a).f(b)$.
Remark that by definition of $g$ and since $S\subset M_{f}$, we have
$$
g(\frac{a}{s}.\frac{b}{t})=\frac{f(ab)}{f(st)}=\frac{f(ab)}{f(s)f(t)}\leq
\frac{f(a)}{f(s)}.\frac{f(b)}{f(t)}=g(a/s).g(b/t).
$$
We also have $g(a/s+b/t)=g(\frac{at+bs}{st})=g(at+bs)/g(st)$ and
$$g(at+bs)/g(st)\leq g(at)/g(st)+g(bs)/g(st)=g(a)/g(s)+g(b)/g(t).$$
This shows that $g$ is a halo morphism.
\end{proof}

\begin{corollary}
Let $A$ be a halo and $S\subset A$ be a multiplicative
subset.
The localized semiring $A_{S}$ is equipped with a natural
halo structure such that if $f:A\to B$ is a multiplicative halo morphism with
$f(S)\subset B^\times$ then $f$ factorizes uniquely through
the morphism $A\to A_{S}$.
\end{corollary}

\subsection{Tropical halos and idempotent semirings}
\begin{definition}
A halo $A$ is called tropical if it is non-trivial, totally ordered and
$a+b=\max(a,b)$ for all $a,b\in A$.
\end{definition}

If $A$ is a positive totally ordered halo, we denote
$A_{trop}$ the same multiplicative monoid
equipped with its tropical addition
$a+_{trop}b:=\max(a,b)$.  There is a natural
halo morphism $A_{trop}\to A$. Remark that
tropical halos usually don't have the simplification
property
$$x+y\leq x+z\Rightarrow y\leq z.$$

\begin{definition}
Let $A$ be a positive halo. A is called archimedean if
$A$ is not reduced to $\{0,1\}$ and
for all $x>y>0$, there exists $n\in \N$ with
$ny>x$. Otherwise, $A$ is called non-archimedean.
\end{definition}
 
\begin{lemma}
\label{tropical-optimistic}
A tropical halo is positive and non-archimedean.
\end{lemma}
\begin{proof}
Let $A$ be a tropical halo. This implies that $0\neq 1$
because $A$ is non-trivial.
Suppose that $1\leq 0$ in $A$.
Then $1+0:=\max(1,0)=0\neq 1$, which
is a contradiction with the fact that $A$ is a semiring.
If $A$ is reduced to $\{0,1\}$, then it is non-archimedean.
Now suppose that $A$ is not reduced to $\{0,1\}$.
If $a>b>0$ then $a>n.b=b$ for all $n\in \N$ so that $A$ is
non-archimedean.
\end{proof}

We will now show that tropical halos and idempotent semirings
are related.
\begin{definition}
A semiring $A$ is called idempotent if $a+a=a$ for all $a\in A$.
\end{definition}

Let $A$ be an idempotent semiring. The relation
$$a\leq b\Leftrightarrow a+b=b$$
is a partial order relation on $A$ that gives $A$
a halo structure denoted $A_{halo}$. This will be called
the natural halo structure of the idempotent
semiring.

\begin{example}
The semiring $\R_{+,trop}$ of positive real numbers with usual multiplication
and tropical addition given by $a+_{trop}b=\max(a,b)$ is an idempotent
semiring which is tropical. The semiring $\R_{+,trop}[X]$ of
polynomials is idempotent but not tropical because
$X$ and $1$ cannot be compared in it:
$X+1\neq X$ and $X+1\neq 1$.
\end{example}

\begin{definition}
Let $K$ be a tropical halo and $S$ be a set.
Then the polynomial semiring $K[S]$ is idempotent and is thus
equipped with a natural halo structure. This halo will be called
the halo of polynomials on $K$.
\end{definition}

The following Lemma shows that the order of a tropical halo is
of a purely algebraic nature.

\begin{lemma}
The functor $A\mapsto A_{halo}$ induces an equivalence
of categories between idempotent semirings whose natural order is total
and tropical halos with multiplicative morphisms.
\end{lemma}
\begin{proof}
First remark that if $f:A\to B$ is a semiring morphism between
two idempotent semirings, then $a\leq b$ in $A$
implies $a+b=b$ so that $f(a)+f(b)=f(b)$ and $f(a)\leq f(b)$ in $B$,
which means that $f$ is an increasing map. This shows that
the map $A\mapsto A_{halo}$ is a functor.
If $A$ is a tropical halo, then its underlying semiring is idempotent.
The natural order of this semiring is equal to the given order
and this last one is total. This shows that the functor
$A\mapsto A_{halo}$ is essentially surjective
from the category of idempotent semirings with total natural
order to tropical halos. Let $A$ and $B$ be two tropical halos.
A halo morphism $f:A\to B$ is increasing so that
$f(\max(a,b))=\max(f(a),f(b))$ and $f$ is
a semiring morphism. This shows that the functor
$A\mapsto A_{halo}$ is full. It is also faithful because
$A$ and $A_{halo}$ have the same underlying set.
\end{proof}

\begin{remark}
Usual semirings with their trivial order and tropical halos
are two subcategories of the category $\Halos^m$ of halos
with multiplicative morphisms that share
a common feature: all their morphisms are strictly additive,
i.e., semiring morphisms.
\end{remark}

\begin{definition}
Let $\Gamma$ be a multiplicative totally ordered monoid and
$R_{\Gamma}:=\{0\}\cup \Gamma$. We equip $R_{\Gamma}$
with a tropical halo structure by declaring that
$0$ is smaller than every element of $\Gamma$,
annihilates every element of $\Gamma$ by multiplication,
and that $a+b=\max(a,b)$ for all $a,b\in R_{\Gamma}$.
The halo $R_{\Gamma}$ is called the tropical
halo of $\Gamma$. If $\Gamma$ is a group,
the tropical halo $R_{\Gamma}$ is an aura.
\end{definition}

\begin{example}
Let $\Gamma$ be a totally ordered group and $H\subset \Gamma$
be a convex subgroup (i.e. $g<h<k$ in $\Gamma$ and $g,k\in H$ implies
$h\in H$). Then $\pi_{H}:R_{\Gamma}\to R_{\Gamma/H}$ is a surjective
halo morphism between tropical auras.
\end{example}


\begin{example}
Let $R_{\{1\}}=\{0,1\}$ be the tropical halo on the trivial group.
It is equipped with the order given by $0\leq 1$, idempotent
addition given by $1+1=1$ and usual multiplication. It is the
initial object in the category of positive halos (in which $0\leq 1$),
so in particular in the category of tropical halos.
Indeed, if $A\neq 0$ is such a halo, then the injective map
$f:R_{\{1\}}\to A$ that sends $0$ to $0$ and $1$ to $1$ is
a halo morphism because $0\leq 1$ implies $f(1+1)=f(1)=1\leq 1+1=f(1)+f(1)$.
\end{example}

\begin{definition}
A halo is called trivial if it is reduced to $\{0\}$ or equal to the tropical
halo $R_{\{1\}}$.
\end{definition}

\subsection{Halos with tempered growth}
If $A$ and $R$ are two halos, halo morphisms $|.|:A\to R$ are not easy
to compute in general. We now introduce a condition that can be imposed
on $R$ to make this computation easier. This condition is directly inspired
by Ostrowski's classification of multiplicative seminorms on $\Z$.
\begin{definition}
A halo $R$ has tempered growth if for all non-zero polynomial $P\in \N[X]$,
$$x^n\leq P(n)\textrm{ in }R\textrm{ for all }n\in \N\textrm{ implies }x\leq 1.$$
\end{definition}

\begin{lemma}
A tropical halo has tempered growth.
\end{lemma}
\begin{proof}
Let $R$ be a tropical halo. Then $n=1+_{trop}\cdots+_{trop}1=1$
in $R$ for all $n\in \N$ so that
$P(n)=1$ in $R$ for all $n\in \N$ and $P\in \N[X]$ non-zero. Let $P$
be such a polynomial.
Suppose that $a^n\leq P(n)$ for all $n\in\N$. In particular, we have
$a\leq P(1)=1$ which shows that $R$ has
tempered growth.
\end{proof}

\begin{lemma}
\label{archimedean-tempered}
Let $R$ be a non-trivial totally ordered positive aura in which
$x>y$ implies that there exists $t>0$ such that $x=y+t$.
Suppose moreover that $\N$ injects in the underlying semiring of $R$.
If $R$ is archimedean then $R$ has tempered growth.
\end{lemma}
\begin{proof}
Let $R$ be as in the hypothesis of this Lemma.
Let $P\in \N[X]$ be a non-zero polynomial of degree $d$
and suppose there exists $x>1$ such that $x^n\leq P(n)$ in $R$ for all $n\neq 0$.
By hypothesis, we can write $x=1+t$ with $t>0$ and
$x^n=(1+t)^n=1+nt+\frac{n(n-1)}{2}t^2+\cdots$.
The components of this sum are all positive so that
$\frac{n!}{p!(n-p)!}t^p\leq P(n)$.
Write $P(n)=\sum a_{i}n^i$. Since $R$ is archimedean,
we can choose $n=n.1>\max(a_{i})$. Now since $R$ is a totally ordered
aura, $n^i\leq n^j$ for all $i\leq j$ so that $P(n)=\sum a_{i}n^i\leq (d+1)n^{d+1}$.
We have proved that $\frac{n!}{p!(n-p)!}t^p\leq (d+1)n^{d+1}$ in $R$
for all $n$ big enough in $\N$. If we take $n>2p$ and $p=d+1$,
we get
$$n(n-1)(n-2)\cdots(n-d)t^p=\frac{n!}{(n-p)!}t^p\leq (d+1)p!n^{d+1},$$
where the left product has $d+2$ terms.
Remark now that $n-i\geq n-(d+1)$ implies $\frac{n}{n-i}\leq \frac{n}{n-(d+1)}$
in $R$.
We also have $n>2(d+1)$ implies $n-(d+1)\geq \frac{n}{2}$ in $R$, so that
$\frac{n}{n-i}\leq 2$ for $i\leq d+1$, which shows that
$$(n-p)t^p\leq (d+1)p!2^{d+1}.$$
Since $t>0$ and $R$ is a totally ordered aura, $t^p>0$.
Moreover, the equality $t^p=(d+1)p!2^{d+1}$ for all successive $p$ is
not possible. Indeed, this would give $t=\frac{t^{p+1}}{t^p}=\frac{(p+1)!}{p!}=p+1$
and $t=p+2$ for a convenient $p$ so that
$p+1=p+2$ which is a contradiction with the fact that $\N$ injects in $R$.
We thus have $t^p<(d+1)p!2^{d+1}$ and since $R$ is archimedean,
we know that there exists $n$ big enough such that
$(n-p)t^p>(d+1)p!.2^{d+1}$. This gives a contradiction.
\end{proof}

\begin{corollary}
The aura $\R_{+}$ has tempered growth.
\end{corollary}
\begin{proof}
This follows directly from Lemma \ref{archimedean-tempered}.
We can also give a more direct proof using a little bit of real
analysis. Let $x\in \R_{+}$ and $P\in \N[X]$ be such that
$x^n\leq P(n)$ for all $n\in \N$. Then taking $n$-th root
and passing to the limit, we get
$$x\leq \lim_{n\to \infty} \sqrt[n]{P(n)}=1.$$
\end{proof}

We will see in the next Section some nice examples of
archimedean auras in which $\N$ embeds but that have
non-tempered growth, showing that the hypothesis
of Lemma \ref{archimedean-tempered} are optimal.

\begin{remark}
We know from Lemma \ref{tropical-optimistic} that a tropical halo $A$ is
non-archimedean. This shows that being of tempered growth is not
equivalent to being archimedean.
\end{remark}

\subsection{Lexicographic products}
Let $R_{1},\dots,R_{n}$ be a finite family of positive auras.
Equip $\prod R_{i}$ with its lexicographic order. Remark
that $\prod R_{i}^\times\subset\prod R_{i}$ is a multiplicative
submonoid that is stable by addition because $a,b>0$ in $R_{i}$
implies $a+b>0$. We extend this embedding to
$\{0\}\cup \prod R_{i}^\times$ sending $0$ to $(0,\dots,0)$.
We will denote $\left[\prod\right] R_{i}:=\{0\}\cup \prod R_{i}^\times$
with its halo structure induced by its embedding into $\prod R_{i}$.
This halo is automatically an aura.

\begin{definition}
Let $R_{1},\dots,R_{n}$ be a finite list (i.e. ordered family) of positive auras.
The aura $\left[\prod\right] R_{i}$ is called the lexicographic
product of the family. If $R$ is a positive aura, we denote
$R^{[n]}$ the lexicographic product $\left[\prod_{1,\dots,n}\right] R$.
If $R$ and $S$ are positive aura, we denote $R\lextimes S$
their lexicographic product.
\end{definition}

Remark that if the $R_{i}$ are totally ordered, then so is $\left[\prod\right] R_{i}$.

\begin{lemma}
If $n>1$, the aura $\R_{+}^{[n]}$ is archimedean but it does not
have tempered growth.
\end{lemma}
\begin{proof}
Suppose that $0<(x_{i})<(y_{i})$ in $\R_{+}^{[n]}$.
Then at least $0<x_{1}\leq y_{1}$ in $\R_{+}$ so that
there exists $n\in \N$ such that $nx_{1}>y_{1}$, which implies
$n(x_{i})>(y_{i})$ in $\R_{+}^{[n]}$. This shows that
$\R_{+}^{[n]}$ is archimedean. Now let $a>1$ in $\R$.
Then $(1,\dots,1,a)>(1,\dots,1)$ in $\R_{+}^{[n]}$ but since
$1<n+2$,
$(1,\dots,1,a)^n=(1,\dots,1,a^n)<(n+2,\dots,n+2)=n+2\in \R_{+}^{[n]}$.
This shows that $\R_{+}^{[n]}$ does not have tempered growth.
\end{proof}

\begin{lemma}
\label{trop-lextimes-tempered}
Let $K$ be a tropical aura and $R$ be an aura that has tempered
growth.
The aura $K\lextimes R$ has tempered growth. 
\end{lemma}
\begin{proof}
Let $P\in \N[X]$ be a non-zero polynomial and $(x,y)\in K\lextimes R$
be such that $(x,y)>1=(1,1)$ and $(x,y)^n\leq (P(n),P(n))$
for all $n$. Remark that $P(n)=1\in K$ for all $n$.
$(x,y)^1\leq (P(1),P(1))=(1,P(1))$ implies $x\leq 1$. If $x<1$,
then $(x,y)<(1,1)$ which is a contradiction. If $x=1$ then
$y>1$. Remark that $(x,y)^n=(1,y^n)\leq (P(n),P(n))=(1,P(n))$ for all $n$
implies $y^n\leq P(n)$ for all $n$. Since $R$ has tempered growth,
this means that $y\leq 1$ in $R$, which is a contradiction.
We thus have proved that
$K\lextimes R$ has tempered growth.
\end{proof}

\begin{corollary}
The aura $\R_{+,trop}\lextimes \R_{+}$ has tempered growth.
\end{corollary}

\section{Seminorms, valuations and places}
Some proofs of the forthcoming Sections are very similar to their classical
version, which one can find in E. Artin's book \cite{E-Artin1} and in
Bourbaki \cite{Bourbaki-AC-5-6}, Alg\`ebre Commutative, Chap. VI.

\subsection{Generalizing seminorms and valuations}
\begin{definition}
A generalized seminorm
on a ring $A$ is a halo morphism from
$A$ to a positive totally ordered aura $R$, i.e., a map $|.|:A\to R$
from $A$ to a positive totally ordered semifield $R$ such that
\begin{enumerate}
\item $|1|=1$, $|0|=0$,
\item for all $a,b\in A$, $|ab|\leq |a|.|b|$,
\item for all $a,b\in A$, $|a+b|\leq |a|+|b|$.
\end{enumerate}
A generalized seminorm $|.|:A\to R$ is called
\begin{itemize}
\item square-multiplicative if $|a^2|=|a|^2$,
\item power-multiplicative if $|a^n|=|a|^n$ for all $a\in A$ and all $n\in \N$,
\item tempered if $R$ has tempered growth,
\item non-archimedean if
$$|a+b|\leq \max(|a|,|b|)$$
for all $a,b\in A$,
\item pre-archimedean if
$$|a+b|\leq \max(|2|,1).\max(|a|,|b|)$$
for all $a,b\in A$.
\end{itemize}
\end{definition}

We will often omit ``generalized'' in ``generalized seminorm''.

\begin{remark}
Let $A$ be a ring. A generalized seminorm on
$A$ with values in $\R_{+}$ is exactly a seminorm on $A$ in the usual sense.
A multiplicative generalized seminorm on $A$ with value in a tropical aura
$R_{\Gamma}=\{0\}\cup\Gamma$ is exactly a valuation in Krull's sense
(multiplicative notation).
\end{remark}

Let $A$ be a ring and $|.|:A\to R$ be a generalized seminorm on $A$.
Then $\Ker(|.|)$ is an ideal in $A$. Indeed, if $|a|=0$ and $|b|=0$,
then $|a+b|\leq |a|+|b|=0$ so that $|a+b|=0$. If $|a|=0$ and $b\in A$,
then $|a.b|\leq |a|.|b|=0$. Remark also that $|.|:A\to R$
factorizes though $A/\Ker(|.|)$. Indeed, if $|a|=0$ and $b\in A$,
then $|b|=|b+a-a|\leq |b+a|+|-a|\leq |b+a|+|-1|.|a|=|b+a|$ and $|b+a|\leq |b|+|a|=|b|$, which
shows that $|b+a|=|b|$.
If $|.|$ is multiplicative (resp. square-multiplicative, resp. power-multiplicative) then its kernel
is a prime (resp. square-reduced, resp. reduced) ideal.

\begin{lemma}
\label{square-multiplicative-sign}
A seminorm $|.|$ always fulfills
$|-1|\geq 1$. If it is square-multiplicative, it
moreover fulfills $|-a|=|a|$ for all $a\in A$.
\end{lemma}
\begin{proof}
If $|-1|<1$ then
$1=|1|=|(-1)^2|\leq |-1|^2\leq 1.|-1|=|-1|$ which is
a contradiction.
Suppose now that $|.|$ is square-multiplicative.
We then have $|-1|=1$. Indeed, if $|-1|>1$ then
$1=|1|=|(-1)^2|=|-1|^2\geq |-1|.1=|-1|$
which is a contradiction.
We also have $|-a|=|a|$
for all $a\in A$. Indeed,
we have $|-a|\leq |-1|.|a|=|a|$ and $|a|=|-(-a)|\leq |-a|$
so that $|-a|=|a|$.
\end{proof}

\begin{example}
The map $|.|=|.|_{6,0}:\Z\to R_{\{1\}}:=\{0,1\}$
given by setting $|n|=0$ if $6|n$ and $1$ otherwise
is not multiplicative because $|2.3|=0<|2|.|3|=1$ but it
is power-multiplicative. Similarly, the $6$-adic seminorm
$|.|=|.|_{6}:\Z\to \R_{+}$ that sends a number $n$ to $6^{-\mathrm{ord}_{6}(n)}$
is power-multiplicative but not multiplicative because
$|2.3|_{6}=1/6<|2|_{6}.|3|_{6}=1$.
\end{example}

\begin{remark}
If we stick to $\R_{+}$-valued seminorms, one can show
that power-multiplicative seminorms correspond (through
the supremum construction) to compact subsets of the
space of bounded $\R_{+}$-valued multiplicative seminorms
on the corresponding completion, as explained by Berkovich
in \cite{Berkovich1}, chapter 1.
Moreover, the power-multiplicativity condition can be shown to be equivalent to
the square-multiplicativity condition $|a^2|=|a|^2$ for all element $a$ of the
algebra (by using the spectral radius). The advantage of the square-multiplicative
formulation is that it is defined by a first order logic condition that is easier to
deal with using model theoretic methods. Next Lemma shows that the notion
of pre-archimedean seminorm (which is also adapted to model theoretic
method) allows an easier determination of non-archimedean valuations.
\end{remark}

\begin{lemma}
\label{pre-non-archimedean}
Let $|.|:A\to R$ be a pre-archimedean seminorm, i.e.,
$$|a+b|\leq \max(|2|,1).\max(|a|,|b|)$$
for all $a,b\in A$. Then the following conditions are equivalent:
\begin{enumerate}
\item $|.|$ is non-archimedean,
\item $|n|\leq 1$ for all integer $n\in \N$,
\item $|2|\leq 1$.
\end{enumerate}
Moreover, if $R_{2}^*\subset R^*$ is the convex subgroup generated
by $|2|$, then the induced map $|.|':A\to R_{R^*/R_{2}}$ is a non-archimedean
seminorm.
\end{lemma}
\begin{proof}
The first condition implies the second because of the ultrametric inequality.
The second implies the third.
If $|2|\leq 1$ and $|.|$ is pre-archimedean then $|.|$ is
non-archimedean because $\max(|2|,1)=1$. This shows the equivalence.
Because of the pre-archimedean condition, and since $|2|=|1|$ in
$R_{R^*/R_{2}}$, we conclude that $|.|'$ is non-archimedean.
\end{proof}

We now define two notions of equivalence of halo morphisms.

\begin{definition}
Let $A$ be a halo, let $|.|_{1}:A\to R$ and $|.|_{2}:A\to S$ be two
halo morphisms from $A$ to two given halos.
We say the $|.|_{1}$ is bounded (resp. multiplicatively bounded)
by $|.|_{2}$ and we write $|.|_{1}\leq |.|_{2}$ (resp. $|.|_{1}\leq_{m}|.|_{2}$)
if for all $a,b\in A$ (resp. for all $a,b,c\in A$),
$$
\begin{array}{c}
|a|_{2}\leq |b|_{2}\Rightarrow |a|_{1}\leq |b|_{1}\\
\textrm{(resp. }|a|_{2}.|c|_{2}\leq |b|_{2}\Rightarrow |a|_{1}.|c|_{1}\leq |b|_{1}\textrm{)}.
\end{array}
$$
We say that $|.|_{1}$ is equivalent (resp. multiplicatively equivalent) to $|.|_{2}$ if
$$
\begin{array}{c}
|.|_{1}\leq |.|_{2}\textrm{ and }|.|_{2}\leq |.|_{1}\\
\textrm{(resp. }|.|_{1}\leq_{m} |.|_{2}\textrm{ and }|.|_{2}\leq_{m} |.|_{1}\textrm{)}.
\end{array}
$$
\end{definition}

Remark that multiplicative equivalence is stronger than equivalence
and allows one to transfer multiplicativity properties between equivalent
seminorms. In particular, if two seminorms $|.|_{1}:A\to R_{1}$
and $|.|_{2}:A\to R_{2}$ are multiplicatively equivalent, then
their multiplicative subsets in $A$ are equal. If we denote it $M$,
by Proposition \ref{prop-localization-multiplicative}, we thus get
factorizations $|.|_{1}:A_{M}\to R_{1}$ and $|.|_{2}:A_{M}\to R_{2}$
that remain multiplicatively equivalent.

The proof of the following theorem, that was our main motivation to introduce
the notion of tempered seminorm, is a refinements of Artin's proof of Ostrowski's
classification of absolute values on $\Z$.
\begin{theorem}
\label{tempered-pre-archimedean}
Let $|.|:A\to R$ be a tempered power-multiplicative seminorm on a ring $A$.
Then $|.|$ is pre-archimedean,
i.e. fulfills
$$|a+b|\leq \max(|2|,1).\max(|a|,|b|)$$
for all $a,b\in A$. If moreover $|2|>1$ then $|.|_{|\Z}$ is multiplicatively
equivalent to the archimedean seminorm
$|.|_{\infty}:\Z\to \Q_{+}$ given by $|n|_{\infty}=\max(n,-n)$.
\end{theorem}
\begin{proof}
We can first suppose that $A=\Z$.
Let $n,m>1$ be two natural numbers. We may write
$m^s=a_{0}+a_{1}n+a_{r}n^{r(s)}$ where $a_{i}\in \{0,1,\dots,n-1\}$
and $n^{r(s)}\leq m^s$.
More precisely, $r(s)$ is the integral part $\left[s.\frac{\log m}{\log n}\right]$
of $s.\frac{\log m}{\log n}$,
so that there exists a constant $c_{m,n}\in \N$ such that
$r(s)\leq s.c_{m,n}$ for all $s\in \N$. In fact, we can choose
$c_{m,n}=1+\left[\frac{\log m}{\log n}\right]$.
Now remark that $|a_{i}|=|1+\cdots+1|\leq a_{i}.|1|\leq n$
for all $i$. We now use that $|.|$ is power-multiplicative, so that
$$|m|^s=|m^s|\leq \sum_{i=0}^{r(s)}|a_{i}|.|n|^i\leq \sum|a_{i}|\max(1,|n|)^{r(s)}\leq n(1+r(s))\max(1,|n|)^{r(s)}.$$
Since $r(s)\leq s.c_{m,n}$, we get finally
$$|m|^s\leq n(1+sc_{m,n})\max(1,|n|)^{sc_{m,n}}.$$
This gives
$$\left(\frac{|m|}{\max(1,|n|)^{c_{m,n}}}\right)^s\leq n(1+sc_{m,n})=P(s),$$
for $P(X)=n(1+Xc_{m,n})\in \N[X]$. Since $R$ has tempered growth,
this implies that
$\left(\frac{|m|}{\max(1,|n|)^{c_{m,n}}}\right)\leq 1$, i.e.,
$$|m|\leq \max(1,|n|)^{c_{m,n}}$$
for $c_{m,n}\in \N$.

First suppose that $|2|\leq 1$. Then, applying
the above inequality with $n=2$, we get $|m|\leq 1$ for all $m\in \Z$.
If $a,b\in A$, we have
$$
|a+b|^s=|(a+b)^s|\leq \sum_{i=0}^s|\binom{s}{i}|.\max(|a|,|b|)^s\leq
(s+1)\max(|a|,|b|)^s
$$
because $\binom{s}{i}$ is an integer so that $|\binom{s}{i}|\leq 1$.
Since $R$ has tempered growth, this inequality implies
$$|a+b|\leq \max(|a|,|b|).$$

Now suppose that $1<|2|$. Then applying the above
inequality with $m=2$, we get $1<|2|\leq \max(1,|n|)^{c_{m,n}}$ so
that $1<\max(1,|n|)$ and $|n|>1$ for all non-zero $n\in \Z$.

Now suppose that $m>n>1$. Then $\log(m)>\log(n)$ so that
$\frac{\log m}{\log n}>1$ and
$c_{m,n}=1+\left[\frac{\log m}{\log n}\right]\geq 2$. We also have
$c_{n,m}\leq 1$, which implies that $c_{n,m}=1$,
so that $|n|\leq |m|^{c_{n,m}}=|m|$. We thus have proved
that $|.|:\N\to R$ is an increasing map for the usual order on $\N$.

As before, if $a,b\in A$, we have
$$
|a+b|^s=|(a+b)^s|\leq \sum_{i=0}^s|\binom{s}{i}|.\max(|a|,|b|)^s.
$$
We recall for reader's convenience that,
since the binomial coefficient $\binom{s}{i}$ counts the number
of parts of cardinal $i$ in a set of cardinal $s$ and $2^s$ is the cardinal
of the set of parts of a set of cardinal $s$, we have
$\binom{s}{i}\leq 2^s$.
Since $|.|:\N\to R$ is increasing, we get $|\binom{s}{i}|\leq |2^s|$
and
$$|a+b|^s\leq (s+1)(|2|.\max(|a|,|b|))^s,$$
and since $R$ has tempered growth, this inequality implies
$$|a+b|\leq |2|.\max(|a|,|b|).$$
Together with what we showed at the beginning, this implies that
$|.|$ is pre-archimedean.

It remains to prove that it is injective.
Since $\frac{\log n}{\log m}<1$, there exists a rational number
$p/q$ such that $\frac{\log n}{\log m}<p/q<1$.
Now if we work in the usual real numbers, we have the inequality
$n=m^{\frac{\log n}{\log m}}\leq m^{p/q}$
so that $n^q\leq m^p$ in the ordered set $\N$ of integers $\N$. 
Since $|.|$ is increasing, we get
$|n|^q\leq |m|^p$ with $p<q$. Changing $p/q$ to $(up)/(uq)$
with $u>1$, we can suppose that $q-p>1$.
Now suppose that $|m|=|n|$. This implies
$|n|^q\leq |n|^p$ with $p<q$ and since
$|n|>1$, this gives
$|n|^{q-p}\leq |n|$ with $q-p>1$.
But since $|n|>1$, this gives a contradiction.

Now suppose that for $m>n>1$, we have $0<|mn|<|m|.|n|$.
This implies that for all $q\in\N-\{0\}$, we have
$$1<\left(\frac{|m|.|n|}{|mn|}\right)^q.$$

Remark that we have $mn=m^{\frac{\log n}{\log m}+1}$,
so that for all rational number $p/q$ such that
$\frac{\log n}{\log m}<p/q<1$, we have
$n\leq m^{p/q}$ so that $n^q\leq m^p$
and $|m|^q.|n|^q\leq |n|^{p+q}$. We also
have $p<q$ so that $(mn)^q\geq (mn)^p=m^pn^p\geq n^{p+q}$
and
$$1<\left(\frac{|m|.|n|}{|mn|}\right)^q\leq \frac{|n|^{p+q}}{|n^{p+q}|}\leq 1$$
(by power-multiplicativity) which is a contradiction. We thus get
that $|m|.|n|\leq |mn|$ and $|.|_{|\Z}$ is multiplicative and
multiplicatively equivalent to $|.|_{\infty}$.
\end{proof}

\subsection{Places}
\begin{definition}
A place of a ring is a multiplicative equivalence class of
generalized seminorm. If $x$ is a place
of a ring $A$, we denote $|.(x)|:A\to R$ a given representative
of $x$.
\end{definition}

If the representative $|.(x)|$ of a given place $x$ of $A$ is multiplicative,
then all other representatives, being multiplicatively equivalent to it,
will also be multiplicative. This shows that our notion of place of
a ring generalizes the classical notion.

The $p$-adic valuation $|.|_{p,trop}:\Z\to R_{p^\Z}$ and
the $p$-adic seminorm $|.|_{p}:\Z\to \R_{+}$ are equivalent
tempered multiplicative seminorms. They thus represent the
same place of $\Z$.

The use of non-multiplicative seminorms in analytic geometry
is imposed by the central role played by the notion of uniform convergence
on compacts in the theory of complex analytic functions. For example,
if $K\subset \C$ is a compact subset, we want the seminorm
$$|.|_{\infty,K}:\C[X]\to \R_{+}$$
given by $|P|_{\infty,K}:=\sup_{x\in K}|P(x)|_{\C}$
with $|.|_{\C}:\C\to \R_{+}$ the usual complex norm
to be a (square-multiplicative) place of $\C[X]$.

\section{Harmonious spectra}
We now want to define a notion of spectrum of a ring
that combines the valuation (or Zariski-Riemann) spectrum
with the seminorm spectrum. It will be called the harmonious
spectrum. We define various versions of this space that
are adapted to the various problem we want to solve.

\subsection{Definition}
\begin{definition}
Let $A$ be a ring. We define various spaces of
seminorms on $A$.
\begin{enumerate}
\item The multiplicative harmonious spectrum of $A$
is the set $\Speh^{m}(A)$ of multiplicative tempered places
on $A$.
\item The power-multiplicative harmonious spectrum of
$A$ is the set $\Speh^{pm}(A)$ of tempered power-multiplicative places
on $A$.
\item The pre-archimedean square-multiplicative harmonious
spectrum of $A$ is the set $\Speh^{pasm}(A)$ of pre-archimedean square-multiplicative
places of $A$.
\item The Krull valuation spectrum of $A$ is the set
$\Spev(A)=\Speh^v(A)$ of equivalence classes
of Krull valuations, i.e., multiplicative tropical seminorms $|.|:A\to R_{\Gamma}$.
\end{enumerate}
For $\bullet\in \{v,m,pm,pasm\}$, the topology on $\Speh^\bullet(A)$ is generated by
subsets of the form
$$
U\left(\frac{a}{b}\right)=
\{x\in \Speh^\bullet(A)|\,
|a(x)|<|b(x)|,\,d|b\Rightarrow d\textrm{ is multiplicative for }|.(x)|\}.
$$
\end{definition}

The spaces $\Speh^{pasm}(A)$ and $\Spev(A)$ can be studied by
model theoretic methods as the ones used in \cite{Prestel} and \cite{Huber-Knebusch}
because they are defined in the setting of first order logic.

\begin{remark}
There are also good reasons to use the topology generated by subsets of
the form
$$
U\left(\frac{a}{b}\right)=
\{x\in \Speh^\bullet(A)|\,
|a(x)|\leq|b(x)|\neq 0,\,d|b\Rightarrow d\textrm{ is multiplicative for }|.(x)|\}.
$$
These are better compatible with Huber's version of analytic spaces
and don't seem to be so uncompatible with archimedean analytic
geometry as was explained to us by Huber (private mail).
\end{remark}

\begin{remark}
The topology on $\Speh^m(A)$ is generated by subsets
of the form $U\left(\frac{a}{b}\right)=\{x\in \Speh^\bullet(A)|\,|a(x)|<|b(x)|\}$
because the multiplicativity condition on $b$ is automatic.
We can also see that in any case, the topology on $\Speh^\bullet(A)$
has rational domains of the form
$$
R\left(\frac{a_{1},\dots,a_{n}}{b}\right)=\{x\in \Speh^\bullet(A)|\,
|a_{i}(x)|<|b(x)|,\,d|b\Rightarrow d\textrm{ is multiplicative for }|.(x)|\}
$$
as a basis. Indeed, the intersection of two rational domains
$R\left(\frac{f_{1},\dots,f_{n}}{h}\right)$ and $R\left(\frac{g_{1},\dots,g_{m}}{k}\right)$
is the rational domain
$R\left(\frac{f_{1}k,\dots,f_{n}k,g_1h,\dots,g_mh}{hk}\right)$.
\end{remark}

Let $\Mc(A)$ denote the set of multiplicative $\R_{+}$-valued seminorms,
equipped with the coarsest topology that makes the maps $x\mapsto |a(x)|$
for $a\in A$ continuous.

We recall the following for reader's convenience (see \cite{Poineau1}
for details on spectrally convex subsets).
\begin{definition}
Let $V$ be a compact subset of $\Mc(A)$. Denote $\|.\|_{\infty,V}$
the supremum of all seminorms in $V$ and $\Bc(V)$ the completion
of the localization of $A$ by the set of elements of $A$ that are non-zero
in a neighborhood of $V$.
We call $V$ spectrally convex if the natural map
$$\Mc(\Bc(V))\to \Mc(A)$$
has image $V$.
\end{definition}

\begin{proposition}
\label{sm-real-valued-spectrum}
There is a natural bijection between the set of $\R_{+}$-valued
square-multiplicative seminorms on $A$ and the set of spectrally convex
compact subsets of $\Mc(A)$.
\end{proposition}
\begin{proof}
The fact that power-multiplicative $\R_{+}$-valued seminorms correspond
to compact subsets of the topological space of bounded multiplicative
seminorms on a Banach algebra is explained in \cite{Berkovich1}, Section 1.2.
To a given real-valued power-multiplicative seminorm $|.|:A\to \R_{+}$,
one can associate the compact subset of $\Mc(A)$ given by all multiplicative
seminorms $|.|':A\to \R_{+}$ bounded by it, i.e., such that
$|.|'\leq |.|$.
By using the spectral radius, one shows that every square-multiplicative
$\R_{+}$-valued seminorm is automatically power-multiplicative.
\end{proof}

By applying Theorem \ref{tempered-pre-archimedean} and Proposition
\ref{sm-real-valued-spectrum}, we get a map from the set
of compact subsets of $\Mc(A)$ to $\Speh^{pasm}(A)$ that gives
a bijection between spectrally convex subsets of $\Mc(A)$ and points
in $\Speh^{pasm}(A)$ given by $\R_{+}$-valued seminorms.

\begin{lemma}
\label{valuative-harmonious}
The Krull valuation spectrum $\Spev(A)$
embeds in $\Speh^{m}(A)$, and
can be defined in this space as the subspace
$$
\Spev(A)=\{x\in \Speh^m(A)|\,|2(x)|\leq 1\}=
\{x\in \Speh^m(A)|\,\forall n\in \N,\,|n(x)|\leq 1\}.
$$
\end{lemma}
\begin{proof}
First remark that we know by Theorem \ref{tempered-pre-archimedean} that
all seminorms in $\Speh^{m}(A)$ are pre-archimedean, and by Lemma
\ref{pre-non-archimedean}, the hypothesis implies that they are non-archimedean,
i.e., fulfill that for all $a,b\in A$,
$$|a+b|\leq \max(|a|,|b|)$$
if and only if $|2|\leq 1$, and this is also equivalent to
$|n|\leq 1$ for all $n\in \N$.
The subset described above is exactly $\Spev(A)$.
\end{proof}

We have the following natural diagram of continuous maps.
$$
\xymatrix{
			&\Spev(A)\ar[d]\ar[r]& \Speh^{pm}(A)\ar[d]&\\
\Mc(A)\ar[r] 	& \Speh^m(A) \ar[r] 	& \Speh^{pasm}(A)}
$$

\begin{remark}
The spaces $\Speh^m(A)$ and $\Speh^{pm}(A)$ are defined
by quantifying on integers (because of the temperation and power-multiplicativity
hypothesis).
The spaces $\Speh^{pasm}(A)$ and $\Spev(A)$ have the advantage of being
defined in the setting of first order logic. This makes them
quite well adapted to model theoretic methods. Unfortunately,
we are not able to use them directly to define analytic spaces.
\end{remark}

\subsection{Comparison with Huber's retraction procedure}
\label{huber}
Roland Huber explained to us the following results, which are of great interest for
our study. He also kindly authorized us to include his original ideas in our article,
ideas which are completely from his own mind, and are localized in this subsection.

If $x=|.(x)|:A\to R$ is a multiplicative seminorm on a ring $A$, we denote
$\Gamma_{x}$ the totally ordered group $|K(x)|-\{0\}$ of non-zero
elements in the image of the extension of $x$ to the fraction field
$K(x)$ of $A/\Ker(|.(x)|)$ where $\Ker(|.(x)|)$ is the prime ideal of element of
$A$ whose image by $|.(x)|$ is $0$.

Huber considers the space $\Speh^{pam}(A)$ of pre-archimedean multiplicative
seminorms (called separated quasi-valuations by him) on a ring $A$
equipped with the topology generated by the sets
$\{x\in \Speh^{pam}(A)|\,|a(x)|\leq |b(x)|\neq 0\}$, $a,b\in A$. This space
is defined in the setting of first order logic and can thus be studied with
model theoretic tools. It is spectral and the boolean algebra of constructible
subsets is generated by the above subsets.

The subsets of seminorms $x$ such that $|2(x)|>1$ (resp. $|2(x)|\leq 1$)
is called the archimedean
(resp. non-archimedean) subset of $\Speh^{pam}(A)$.
It is a closed (resp. open) constructible
subset of $\Speh^{pam}(A)$ denoted by $\Speh^{pam}_{a}(A)$ (resp.
$\Speh^{pam}_{na}(A)$), and there is an equality of topological spaces
$$\Speh^{pam}_{na}(A)=\Speh^v(A)=\Spev(A).$$

For the purpose of composing
seminorms, Huber defines a subset $\Speh^{pam}_{sep}(A)$ of $\Speh^{pam}(A)$
composed of separated seminorms. These are seminorms
$|.(x)|:A\to R$ such that for every $\gamma\in \Gamma_{x}$ with $\gamma>1$,
there is some $n\in\N$ with $\gamma^n>|2|$ (in particular, every non-archimedean
seminorm in there is quasi-separated).

Huber remarked that the natural map
$$\Speh^m(A)\to \Speh^{pam}_{sep}(A)$$
from tempered multiplicative seminorms (in the sense of our article) to
pre-archimedean multiplicative and separated seminorms is a bijection.
Moreover, we have an equality of topological spaces
$$\Speh^{pam}_{sep,na}(A)=\Speh^{pam}_{na}(A)=\Spev(A)$$
as said above.

For every $x\in \Speh^{pam}(A)$,  put
$$
\Delta_{x}=\{\gamma\in \Gamma_{x}|\,(\max(1,|2(x)|)^{-1}\leq \gamma^n\leq \max(1,|2(x)|),
\textrm{ for every }n\in \N\}.
$$
Then $\Delta_{x}$ is a convex subgroup of $\Gamma_{x}$. If $x$ is
non-archimedean then $\Delta_{x}=\{1\}$. If $x$ is archimedean,
then $\Delta_{x}$ is the greatest convex subgroup of $\Gamma_{x}$
that does not contain $|2(x)|$. We have a retraction from $\Speh^{pam}(A)$
onto its subset $\Speh^{pam}_{sep}(A)$,
$$r:\Speh^{pam}(A) \to \Speh^{pam}_{sep}(A)$$
given by $x\mapsto x/\Delta_{x}$.
The quotient topology $\tau_{sep}$ of the given topology on $\Speh^{pam}(A)$,
called the retraction topology on $\Speh^{pam}_{sep}(A)$,
gives a very nice combination of the closed inequalities topology
on the non-archimedean part with the open inequalities topology
on the archimedean part.

To be more precise, let us describe the open subset of $\tau_{sep}$.
For every $q\in \Q$, define a function $\lambda_{q}$ on $\Speh^{pam}(A)$
by setting $\lambda_{q}(x)=(\max(1,|2(x)|))^q\in \Gamma_{x}\otimes\Q$. For
all $a,b\in A$, the set
$D(a,b)_{q}:=\{x\in \Speh^{pam}(A)|\,|a(x)|\leq \lambda_{q}(x).|b(x)|\neq 0\}$
is open and constructible in $\Speh^{pam}(A)$. Put
\begin{eqnarray*}
U(a,b)_{q} &  := & \cup_{t\in \Q,t<q}D(a,b)_{t}\\
U_{sep}(a,b)_{q}&:=&U(a,b)_{q}\cap \Speh^{pam}_{sep}(A).
\end{eqnarray*}
Then $U(a,b)_{q}$ is open in $\Speh^{pam}(A)$ and
$U(a,b)_{q}=r^{-1}(U_{sep}(a,b)_{q})$. Hence
$U_{sep}(a,b)_{q}$ is open in $\tau_{sep}$. The non-archimedean
part of $U_{sep}(a,b)_{q}$ is equal to $\{x\in \Spev(A)||a(x)|\leq |b(x)|\neq 0\}$, i.e.,
a standard non-archimedean open subset. The archimedean part
of $U_{sep}(a,b)_{q}$ is equal to
$\{x\in \Speh^{pam}_{sep,a}(A)||a(x)|<\lambda_{q}(x).|b(x)|\}$.

Huber further shows that the topology on the archimedean compotent
$\Speh^{pam}_{sep,a}(A)$ is generated by the sets $U_{sep}(a,b)_{0}$
and that
$\Speh^{pam}_{sep,na}(A)$ is a dense open subset of $\Speh^{pam}_{sep}(A)$.
This implies that the topology on the non-archimedean (resp. archimedean) part of
$\Speh^{pam}_{sep}(A)$ is generated by subsets of the form
$$
\begin{array}{c}
\{x|\,|a(x)|\leq |b(x)|\neq 0\}\\
\textrm{(resp. }\{x|\,|a(x)|<|b(x)|\}\textrm{),}
\end{array}
$$
for $a,b\in A$, so that the retraction topology combines nicely the open inequalities
topology of archimedean analytic geometry with the closed inequalities topology
of non-archimedean analytic geometry.

This shows that the space $\Speh^{pam}(A)$ and its relations to global analytic
geometry deserve to be further studied.

\subsection{The multiplicative harmonious spectrum of $\Z$}
We will now show that the harmonious spectrum of $\Z$
is very close to the previously known spectrum of $\Z$.
\begin{lemma}
\label{finite-field-trivial}
Let $K$ be a finite field and $|.|:K\to R$ be a
multiplicative seminorm. Then $|.|$ is multiplicatively
equivalent to the trivial seminorm $|.|_{0}:K\to R_{\{1\}}=\{0,1\}$.
\end{lemma}
\begin{proof}
If $n$ is the order of $K^\times$, then $x^n=x$ for all $x\in K^\times$.
We can suppose $n>1$.
Let $x\in K^\times$ and suppose that $|x|\neq 1$.
We always have $|x|^n=1$.
If $|x|<1$, then $|x|^{n-1}\leq 1$ so that $|x|^n\leq |x|$. But $|x|^n=1$,
which implies that $1\leq|x|<1$. This is a contradiction. If we suppose
$|x|>1$, we also arrive to $1\geq |x|>1$. This shows that
$|x|=1$.
\end{proof}

\begin{lemma}
\label{tropical-rational}
Let $|.|:\Q\to R$ be a non-archimedean multiplicative seminorm on $\Q$ with
trivial kernel. Then $|.|$ is either equivalent to $|.|_{p}:\Q\to K_{p^\Z}$ for
a prime number $p$ or to $|.|_{0}:\Q\to K_{\{1\}}=\{0,1\}$.
\end{lemma}
\begin{proof}
We have $|n|=|1+\cdots+1|\leq 1$. If $|p|=1$ for all primes, then
$|n|=1$ for all $n$ because of unique factorization. This implies
that $|.|$ is equivalent to $|.|_{0}$. Suppose now that
there exists a prime $p$ such that $|p|<1$.
The set $\pfk=\{a\in \Z|\,|a|<1\}$ is an ideal of $\Z$ such
that $p\Z\subset \pfk\neq \Z$. Since $p\Z$ is a maximal
ideal, we have $\pfk=p\Z$. If now $a\in \Z$ and $a=bp^m$
with $b$ not divisible by $p$, so that $b\notin\pfk$,
then $|b|=1$ and hence $|a|=|p|^m$. Remark now that
$|p|<1$ implies $|p|^{n+1}<|p|^{n}$ for all $n$ so that
the map $|.|:\Q\to R$ factorizes through the tropical
field $K_{|p|^\Z}=1\cup\{|p|^\Z\}\subset R$. We have
thus proved that $|.|$ is equivalent to $|.|_{p}$.
\end{proof}

\begin{proposition}
\label{harmonious-spectrum-of-Z}
There is a natural identification
$$\Speh^{m}(\Z)=\{|.|_{p},|.|_{p,0},\,p\textrm{ prime}\}\cup\{|.|_{\infty}\}$$
of the multiplicative tempered spectrum of $\Z$ with the set
consisting of $p$-adic and $p$-residual seminorms for all
prime ideals $(p)$ (including $(p)=(0)$), and of the archimedean norm.
The natural map
$\Mc(\Z)\to \Speh^{m}(\Z)$ is surjective and can be
described by figure \ref{figure-berko-Z}.
\begin{figure}[h]
\centering
\begin{minipage}{\linewidth}
\includegraphics[width=0.4\textwidth]{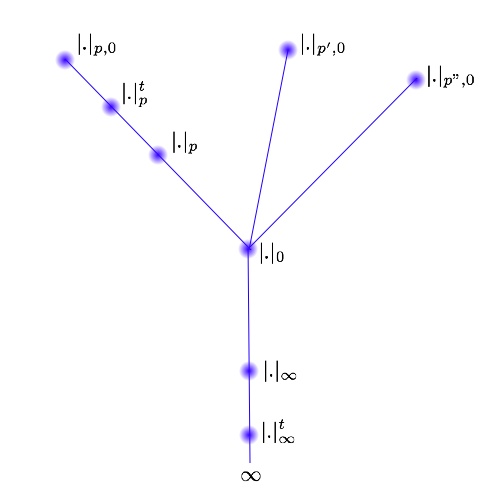}
\end{minipage}%
\begin{minipage}{0.2\linewidth}
\hspace{-9cm}
$\longrightarrow$
\end{minipage}%
\begin{minipage}{\linewidth}
\hspace{-11cm}
\includegraphics[width=0.4\textwidth]{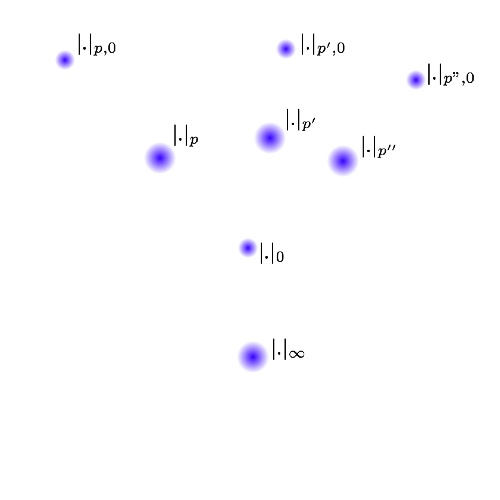}
\end{minipage}
\caption{From Berkovich to Harmonious spectrum.}
\label{figure-berko-Z}
\end{figure}
\end{proposition}
\begin{proof}
Let $|.|=|.(x)|:\Z\to R$ be a tempered multiplicative seminorm
that represents a multiplicative place $x\in \Speh^{m}(\Z)$.
The kernel $\pfk$ of $|.|$ is a prime ideal of $\Z$.
If $\pfk\neq 0$, it is of the form $\pfk=(p)$
for a prime number $p$ and $|.|$ factorizes through the finite
field $\Z/p\Z$. The multiplicative seminorm $|.|:\Z/p\Z\to R$
is equivalent to the trivial seminorm $|.|_{p,0}:\Z/p\Z\to R_{\{1\}}=\{0,1\}$
by Lemma \ref{finite-field-trivial}. Now suppose that
$\pfk=0$ and $|.|$ is a generalized norm.
If $|2|\leq 1$, then by Lemma \ref{valuative-harmonious}, $|.|$ is
non-archimedean. Lemma \ref{tropical-rational} shows
that $|.|$ is equivalent to $|.|_{p}:\Q\to K_{p^\Z}$ or $|.|_{0}:\Q\to K_{\{1\}}$.
If $|2|>1$, then Theorem \ref{tempered-pre-archimedean}
shows that $|.|$ is multiplicatively equivalent to the
usual archimedean seminorm $|.|_{\infty}:\Z\to \Q_{+}$.
All this shows that points of $\Speh^{m}(A)$ are exactly given by
$$\Speh^{m}(A)=\{|.|_{p},|.|_{p,0},|.|_{\infty}\}.$$
The multiplicative seminorm spectrum of $\Z$ is
$$
\Mc(\Z)=
\{|.|_{p}^t,|.|_{p,0},\,p\textrm{ prime},\,t\in [0,\infty[\}\cup\{|.|_{\infty}^t,\,t\in [0,1]\}.
$$
as shown in Emil Artin's book (following Ostrowski) \cite{E-Artin1}.
\end{proof}

\subsection{The multiplicative harmonious affine line over $\Z$}
We will now give a quite complete description of the points
of the multiplicative harmonious affine line over $\Z$.
\begin{definition}
The harmonious affine line over $\Z$, denoted $\A^{1,h}_{\Z}$ is
$\Speh^m(\Z[X])$.
\end{definition}

Recall from Lemma \ref{valuative-harmonious} that
the valuation spectrum $\Spev(\Z[X])\subset \Speh^m(\Z[X])$
is exactly the set of multiplicative seminorms such that $|2|\leq 1$.

\begin{lemma}
\label{Fp-points}
Let $x\in \Speh^m(\Z[X])$ be a seminorm.
If the image of $x$ in $\Speh^m(\Z)$ has a non-zero
kernel $(p)$, then $|.(x)|$ is non-archimedean and
there exists an irreducible polynomial $P\in \F_{p}[X]$
such that
\begin{enumerate}
\item either $|.(x)|$ is equivalent to the trivial seminorm
$|.|_{0,P,p}:\F_{p}[X]/(P)\to R_{\{1\}}=\{0,1\}$,
\item or $|.(x)|$ is equivalent to the $P$-adic seminorm
$|.|_{P}:\F_{p}[X] \to R_{P^\Z}$ given by
$|Q|_{P}=P^{-\mathrm{ord}_{P}(Q)}$.
\end{enumerate}
\end{lemma}
\begin{proof}
Since $|p(x)|<1$ and $|.(x)|$ is pre-archimedean, we get
by Lemma \ref{pre-non-archimedean} that it is non-archimedean. 
We already know that $p.\Z[X]$ is included in the kernel of $|.(x)|$
so that we have a factorization $|.(x)|:\F_{p}[X]\to R$. If
this factorization has a non-trivial kernel, then this
kernel is a prime ideal of $\F_{p}[X]$ generated by
an irreducible polynomial $P$, and by Lemma \ref{finite-field-trivial},
$|.(x)|$ is equivalent to $|.|_{0,P,p}:\F_{p}[X]/(P)\to R_{\{1\}}=\{0,1\}$.
If the factorization $|.(x)|:\F_{p}[X]\to R$ has non-trivial kernel,
then the subset $\pfk=\{P\in \F_{p}[X]|\,|P(x)|<1\}$ is a prime ideal
generated by an irreducible polynomial $P$. We prove similarly
as in Lemma \ref{tropical-rational} that $|.|$ is equivalent to the
multiplicative seminorm $|.|_{P}:\F_{p}[X] \to R_{P^\Z}$ given by
$|Q|_{P}=P^{-\mathrm{ord}_{P}(Q)}$.
\end{proof}

We recall for reader's convenience Huber and Knebusch's
description of the (non-archimedean) valuation spectrum of
a polynomial algebra in \cite{Huber-Knebusch} on an algebraically closed
field $K$, Proposition 3.3.2,
in terms of ultrafilters of discs in $K$. Remark
that $\Spev(\bar{\Q}[T])$ is the quotient of $\Spev(\bar{\Q}[T])$
by the Galois action.

Let $K$ be an algebraically closed field.

\begin{definition}
Let $|.|:K\to R$ be a non-trivial non-archimedean valuation.
A disc of $K$ for $|.|$ is a subset $S\subset K$
of the form
$$
S=B^+(a,\gamma)=\{x\in K|\, |x-a|\leq |\gamma|\}\textrm{ or }
S=B^-(a,\gamma)=\{x\in K|\, |x-a|<|\gamma|\}
$$
for $a,\gamma\in K$.
To a disc of $K$ naturally corresponds a unique subset $\tilde{S}$ of
$\Spev(K[T])$ given by
$$
\tilde{S}=\{x\in \Spev(A)|\, |(X-a)(x)|\leq |\gamma(x)|\}\textrm{ or }
\tilde{S}=\{x\in K|\, |(X-a)(x)|<|\gamma(x)|\}.
$$
The set of discs is denoted $C$.
\end{definition}

\begin{proposition}[Huber-Knebusch]
\label{Huber-Knebusch}
Let $x\in \Spev(K[X])$ be a valuation whose restriction $|.|$ to $\Q$
is non-trivial. There exists a unique filter of discs $F$ of $K$
such that $x\in \tilde{S}$ for all $S\in F$. Distinguishing four cases,
we can give a precise description of $x$.
\begin{enumerate}
\item If there exists $a\in K^*$ such that $F=\{S\in C|\,a\in S\}$,
then $x$ is given by
$$P\mapsto |P(a)|.$$
\item If there exists $a\in K^*$, $\gamma\in K^*$ with $|\gamma|>0$
and $F=\{S\in C|\,B^+(a,\gamma)\subset S\}$, then
$$\left|\sum_{i=0}^n a_{i}(X-a)^i\right|(x)=\max(|a_{i}|.|\gamma|^i|i=0,\dots,n),$$
i.e., $|.(x)|$ is a generalized Gauss valuation.
\item If $\cap_{S\in F}S=\emptyset$ then $x$ is an immediate extension
of $|.|$ to $K(X)$ and can be constructed as follows. Let
$p(X)/q(X)\in K(X)^*$ be given. Choose $S\in F$ which
is disjoint to the zero set of $p(X).q(X)$. Then there exists $\gamma\in K^*$
with $|p(x)/q(x)|=|\gamma|$ for all $x\in S$, and we have
$$|p(X)/q(X)|(x)=|\gamma|.$$
\item Assume that $x$ is not of the previous types. Choose $a\in \cap_{S\in F}S$
and denote $M=\{|\gamma|\in |K^*||\,B^-(a,\gamma)\in F\}$. Then
$M\subset |K^*|$ is a major subset (i.e. if $x\in M$, $y\in |K^*|$ with $x\leq y$
then $y\in M$)
$$\left|\sum_{i=0}^n a_{i}(X-a)^i\right|(x)=\max(|a_{i}|q^i)$$
where the value group is $|K^*|\times q^\Z$ with the ordering
extending the one of $|K^*|$ and such that
$M=\{|\gamma|\in |K^*|,\,q=|X-a|<|\gamma|\}$. More precisely, we have,
depending on $M$, the three following possibilities:
\begin{enumerate}
\item If $M=\emptyset$ then $|\gamma|<q=|X-a|$ for all $\gamma\in K^*$ and
if $a_{n}\neq 0$ then
$$
\left|\sum_{i=0}^n a_{i}(X-a)^i\right|(x)=
|a_{n}X^n|(x)=
|a_{n}|q^n.
$$
\item If $M=|K^*|$ then $|X-a|=q<|\gamma|$ for all $\gamma\in K^*$
and if $a_{i_{0}}\neq 0$ then
$$\left|\sum_{i=i_{0}}^n a_{i}(X-a)^i\right|(x)=|a_{i_{0}}|q^{i_{0}}.$$
\item If $M=\{|\gamma|\in |K^*|,\,|b|<|\gamma|\}$ then
$$\left|\sum_{i=i_{0}}^n a_{i}\left(\frac{X-a}{b}\right)^i\right|(x)=\max(|a_{i}|q^i).$$
\end{enumerate}
\end{enumerate}
\end{proposition}

With help of our temperation condition, a similar classification
result also holds for archimedean tempered power-multiplicative seminorms.

\begin{definition}
Let $|.|:\Q[X]\to R$ be a seminorm. We say that $|.|$ is upper (resp. lower) bounded
if for all $P\in \Q[X]-\{0\}$, there exists $\lambda_{P}\in \Q^*$ (resp. $\mu_{P}\in \Q^*$)
such that $|P|\leq |\lambda_{P}|$ (resp. $|\mu_{P}|\leq |P|$).
\end{definition}

\begin{theorem}
\label{archimedean-harmonious-affine-line}
Let $x\in \Speh^{m}(\bar{\Q}[X])$ be a tempered
multiplicative seminorm such that
$|2|>1$. Distinguishing various cases, we can give the following
description of $x$.
\begin{enumerate}
\item If $|.(x)|$ has non-trivial kernel then there exists $a\in \bar{\Q}$
such that $|.(x)|$ is multiplicatively equivalent to
$$|.(a)|_{\C}:\bar{\Q}[X]\to \R_{+},$$
where $|.|_{\C}$
is the usual complex norm composed with an embedding $\bar{\Q}\to \C$.
\item Otherwise, if $|.(x)|$ is upper and lower bounded then
it is multiplicatively equivalent to an $\R_{+}$-valued seminorm and it extends
to $\C[X]$. More precisely,
there exists a point $a$ of $\C$ (not equal to a point
of $\bar{\Q}$) such that $|.(x)|$ is equivalent to $|.(a)|_{\C}$.
\item If $|.(x)|$ is upper but not lower bounded then
there exists $a\in \bar{\Q}$ such that
$|.(x)|$ is multiplicatively equivalent to $|.|_{2,a}$, where
$$|\sum_{i=i_{0}}^n a_{i}(X-a)^i|_{2,a}=\max(|a_{i}|_{\C}.q^i)=|a_{i_{0}}|_{\C}q^{i_{0}}$$
for $a_{i_{0}}\neq 0$,
where the value halo is $\R_{+}\lextimes R_{q^\Z}$
and $q<r$ for all $r\in \R_{+}$.
\item If $|.(x)|$ is lower but not upper bounded then
$|.(x)|$ is multiplicatively equivalent to $|.|_{2,\infty}$, where
$$|\sum_{i=0}^n a_{i}(X-a)^i|_{2,\infty}=\max(|a_{i}|_{\C}q^i)=|a_{n}|_{\C}q^n$$
for $a_{n}\neq 0$,
where the value halo is $\R_{+}\lextimes R_{q^\Z}$
and $q>r$ for all $r\in \R_{+}$.
\end{enumerate}
\end{theorem}
\begin{proof}
First recall from Proposition \ref{harmonious-spectrum-of-Z} that
since $|2(x)|>1$, $|.(x)|_{|\Z}$ is multiplicative and
equivalent to the usual archimedean norm $|.|_{\C}$.
Since $|.(x)|$ is multiplicative, its kernel is a prime ideal
of $\bar{\Q}[X]$. If this ideal is non-trivial, it is of the form
$(X-a)$ for $a\in \bar{\Q}$. Then $|.(x)|$ factors through
$\bar{\Q}[X]/(X-a)\cong \bar{\Q}$ and it is equivalent
to $|.(a)|_{\C}$. From now on, we suppose that $|.|$ has
trivial kernel.

Now suppose that $|.(x)|$ is both upper and lower bounded.
Then the natural inclusion $i:\bar{\Q}\to \bar{\Q}[X]$ is continuous
for the topology induced on both rings by $|.(x)|$. Indeed,
if $|b(x)|\neq 0$ with $b\in \bar{\Q}[X]$ then there exists
$b'=\mu_{b}\in \bar{\Q}$ such that $|b'(x)|\leq |b(x)|$ so
that $B(0,|b'|)\subset B(0,|b|)$ and $i$ is continuous. This
implies that $i$ extends to $i:\C\to \C[X]$. We
put on $R$ the uniform structure generated
by the sets
$$U_{|\lambda|}=\{(x,y)\in R\times R|\,\max(x,y)\leq\min(y+|\lambda|,x+|\lambda|)\}$$
for $\lambda\in \Q^*$, and denote by $\hat{R}$ the completion of $R$ for
this uniform structure. Then $|.|:\bar{\Q}[X]\to R$ is uniformly continuous
and extends at least to $|.|:\C[X]\to \hat{R}$. We will now replace $R$ by
$\hat{R}$. The fact that $|.|$ is multiplicatively equivalent to an $\R_{+}$-valued
power-multiplicative seminorm follows from forthcoming Lemma
\ref{upper-lower-bounded}. It is well known that such a seminorm
is the composition of the complex norm with evaluation at a point
$a$ of $\C$.

If we suppose that $|.(x)|$ is upper but not lower bounded,
since $|.|$ is multiplicative, the non-archimedean
seminorm $|.|':\bar{\Q}[X]\to K_{R^*/R_{2}}$ (where $R_{2}$
is the convex subgroup generated by $|2|$) is multiplicative,
so that $\pfk=\{x\in \bar{\Q}[X]|\,|x|'<1\}$ is a prime ideal in $\bar{\Q}[X]$.
If it is reduced to $(0)$ then $|.|$ is also lower bounded which
is a contradiction. We thus have $\pfk=(X-a)$ for some $a\in \bar{\Q}$.
Let $P$ be a non-zero polynomial. We want to prove that for all
$x\in \bar{\Q}$, $x\neq a$,
$|X-x|=|x-a|$. First remark that
$\frac{|X-x|}{|x-a|}\leq 1+\frac{|X-a|}{|x-a|}$ so that
$$
\frac{|X-x|^n}{|x-a|^n}\leq
\sum_{p=0}^n|\binom{n}{p}|.\frac{|(X-a)|^p}{|x-a|^p}.
$$
But we know that $|X-a|$ is smaller than any $|\lambda|$ for
$\lambda\in \bar{\Q}^*$ so that
$|\binom{n}{p}|.\frac{|X-a|^p}{|x-a|^p}
\leq 1$. This gives us
$\frac{|X-x|^n}{|x-a|^n}\leq n+1$ so that
$|X-x|\leq |x-a|$. Since $(X-x)\notin \pfk$,
we know that $|X-x|\geq |\mu_{x}|>0$ for some
$\mu_{x}\in \bar{\Q}^*$. This allows us to prove as before
that $|X-x|\geq |x-a|$ so that
we have equality. By induction, if $P=(X-a)^n.u\prod_{x}(X-x)$
with $x\neq a$ and $u\in \bar{\Q}^*$, then
$|P|=|(X-a)^n.u\prod_{x}(x-a)|=|X-a|^n.|u\prod_{x}(x-a)|$.
Since $|.|_{|\bar{\Q}}$ is equivalent to $|.|_{\C}$ and $|X-a|$
is smaller than any $|\lambda|$ for $\lambda\in \bar{\Q}^*$,
we get that $|.|$ is multiplicatively equivalent to
$|.|_{2,a}$ for $a\in \bar{\Q}$.

Now suppose that $|.|$ is lower but not upper bounded.
We will show that if $P=\sum_{i=0}^m a_{i}X^i$ with
$a_{m}\neq 0$ then $|P|=|a_{m}X^m|$.
Since $|.|$ is multiplicative, the non-archimedean
seminorm $|.|':\bar{\Q}[X]\to K_{\R_{+}^*/R_{2}}$ is multiplicative
and extends to $\bar{\Q}(X)$. The subset $A=\{R\in \bar{\Q}[X]|\,|R|'\leq 1\}$
is a valuation ring in $\bar{\Q}(X)$ and
$\pfk=\{R\in \bar{\Q}(X)|\,|R|'<1\}$ is a prime ideal in $A$.
Remark that since $|.|$ is not upper bounded, we know from
forthcoming Lemma \ref{not-upper-bounded-monomials} that for all $x\in \bar{\Q}$
and all $\lambda\in \Q^*$, $|X-x|>|\lambda|$. This implies that $A$ is
a subring of $\bar{\Q}(X)$ that contains the algebra generated by
$\{\frac{1}{X-x}\}_{x\in \bar{\Q}}$. Remark that using the decomposition of
rational fractions in simple parts, we know that a quotient
$\frac{P}{Q}$ of two polynomials $P$ and $Q$ is in $A$ if and only
if $\deg(P)\leq\deg(Q)$. Indeed, if $\frac{P}{Q}\in A$, we can write
$$\frac{P}{Q}=R_0+\sum_{i=1}^m \frac{c_i}{Q_i}$$
with $c_i\in \bar{\Q}^*$, $Q_i\in \bar{\Q}[X]$ of non-zero degree and $R_0\in\bar{\Q}[X]$ of degree equal to
$\deg(P)-\deg(Q)$ (saying that $\deg(0)=-\infty$). If $\deg(P)>\deg(Q)$
then $\deg(R_0)>0$ so that $|R_0|>|\lambda|$ for all $\lambda\in \bar{\Q}^*$.
But since the $Q_i$'s are of non-zero degree, we have
$\left|\frac{c_i}{Q_i}\right|'\leq 1$. Now remark that
$$|R_0|'=\left|\frac{P}{Q}-\sum_{i=1}^m \frac{c_i}{Q_i}\right|'\leq
\max\left(\left|\frac{P}{Q}\right|',\left|\frac{c_i}{Q_i}\right|'\right).$$
This implies that $|R_{0}|'\leq 1$ so that $|R_{0}|$ is bounded by
some $|\lambda|$ for $\lambda\in \bar{\Q}^*$,
which contradicts the hypothesis. We thus have shown that
$$A=\left\{\frac{P}{Q}\in \bar{\Q}(X)|\deg(P)\leq \deg(Q)\right\}.$$
Now denote $P=\sum_{i=0}^m a_{i}X^i$ with $a_{m}\neq 0$
and $S=P-a_{m}X^m=\sum_{i=0}^{m-1}a_{i}X^i$. 
Consider the inequality
$$
\left(\frac{|P|}{|a_{m}X^m|}\right)^r\leq
\left(1+ \frac{|S|}{|a_{m}X^m|}\right)^r\leq
\sum_{j=0}^r|\binom{r}{j}|\frac{|S|}{|a_{m}X^m|}.
$$
Since $\deg(S)<m=\deg(a_{m}X^m)$, we have $|S|'<|a_{m}X^m|'$
which means that for all $\lambda\in \Q^*$,
$|S|<|a_{m}X^m|.|\lambda|$. This implies in particular that
$|\binom{r}{j}|\frac{|S|}{|a_{m}X^m|}\leq 1$ so that
$$\left(\frac{|P|}{|a_{m}X^m|}\right)^r\leq r+1$$
for all $r\in \N$
and since $R$ has tempered growth, we get
$|P|\leq |a_{m}X^m|$.
Similarly, using that $a_{m}X^m=P-S$, we obtain the inequality
$$
\left(\frac{|a_{m}X^m|}{|P|}\right)^r\leq
\left(1+ \frac{|S|}{|P|}\right)^r\leq
\sum_{j=0}^r|\binom{r}{j}|\frac{|S|}{|P|},
$$
and since $\deg(S)<\deg(P)$, we get that
$|P|\geq |a_{m}X^m|$ so that
$$|\sum_{i=0}^m a_{i}X^i|=|a_{m}X^m|.$$
This shows that $|.|$ is multiplicatively equivalent to $|.|_{2,\infty}$.
\end{proof}

\begin{lemma}
\label{upper-lower-bounded}
Let $A$ be a $\C$-algebra and $|.|:A\to R$ be a tempered
power-multiplicative seminorm on $A$ whose restriction to $\C$ is multiplicatively
equivalent to the usual archimedean norm $|.|_{\C}$.
If $|.|$ has trivial kernel and $|.|$ is both upper and lower bounded, then
$|.|$ is equivalent to an $\R_+$-valued seminorm.
\end{lemma}
\begin{proof}
Suppose that $\Ker(|.|)=(0)$ and $|.|$ is both upper and
lower bounded.
Let $P\in A$ be a non-zero element. Then there exist
$\lambda_P,\mu_P\in \C^*$ such that $0<|\lambda_P|\leq |P|\leq|\mu_P|$.
Let $\mu_\infty\in \C^*$ be such that
$|\mu_\infty|_\C=\sup\{|\mu_P|_\C,\,|\mu_P|\leq |P|\}$.
By construction, if $|\lambda_P|\geq |P|$, $|\mu_\infty|\leq |\lambda_P|$.
Now let $\lambda_\infty\in \C^*$ be such that
$|\lambda_\infty|_\C=\inf\{|\lambda_P|_\C,\,|\lambda_P|\geq |P|\}$.
By construction, if $|\mu_P|\leq |P|$, $|\lambda_\infty|\geq |\mu_P|$.
So we have $|\mu_\infty|\leq |\lambda_\infty|$. If $|\mu_\infty|<|\lambda_\infty|$,
then there exists $x\in \C^*$ such that $|\mu_\infty|<|x|<|\lambda_\infty|$.
If $|P|<|x|$, then $|x|\leq|\lambda_\infty|$ which is a contradiction.
Similarly, if $|P|>|x|$, we get a contradiction. If
$|P|=|x|$, we also get a contradiction. This shows that $|\mu_\infty|=|\lambda_\infty|$.
By definition of these two, we also get $|P|=|\mu_\infty|=|\lambda_\infty|$, so that
there exists $\lambda_P\in \C^*$ such that $|P|=|\lambda_P|$.
To sum up, for all $P\in A$, there exist $\lambda\in \C^*$ such
that $|P|=|\lambda_{P}|$, with
$|\lambda_{P}|_{\infty}=\inf\{|\lambda|_{\infty},|P|\leq |\lambda_{\infty}|\}=
\sup\{|\mu|_{\infty},|P|\leq |\mu_{\infty}|\}$.
We define $|P|_{1}:=|\lambda_{P}|_{\infty}$. Since $|.|_{|\C}\sim |.|_{\C}$,
this is well defined. Moreover, it is a seminorm.
Recall that we have $|P+Q|\leq |2|.\max(|P|,|Q|)$ so that
$|\lambda_{P+Q}|\leq |2|.\max(|\lambda_{P}|,|\lambda_{Q}|)
=\max(|2\lambda_{P}|,|2\lambda_{Q}|)$ because
$|.|$ is multiplicative on $\C$. But this implies
$|P+Q|_{1}\leq 2\max(|P|_{1},|Q|_{1})$. The fact that
one can deduce from this inequality that
$|P+Q|_{1}\leq |P|_{1}+|Q|_{1}$  
was known to E. Artin and can be found in \cite{E-Artin1}, Theorem 3.
We have $|PQ|\leq |P|.|Q|$ so that $|\lambda_{PQ}|\leq |\lambda_{P}|.|\lambda_{Q}|$
and since $|.|$ is multiplicative on $\C$, this implies
$|\lambda_{PQ}|\leq |\lambda_{P}\lambda_{Q}|$ and
$|PQ|_{1}\leq |P|_{1}.|Q|_{1}$.
If $|P|.|Q|\leq |PQ|$ then $|\lambda_{P}|.|\lambda_{Q}|\leq |\lambda_{PQ}|$
and since $|.|$ is multiplicative on $\C$, this implies
$|\lambda_{P}\lambda_{Q}|\leq |\lambda_{PQ}|$ so
that
$$
|P|_{1}.|Q|_{1}=|\lambda_{P}|_{\C}.|\lambda_{Q}|_{\C}=
|\lambda_{P}.\lambda_{Q}|_{\C}\leq |\lambda_{PQ}|_{\C}=|PQ|_{1}.
$$
This shows that $|.|$ and $|.|_{1}$ are multiplicatively equivalent and
we conclude that $|.|_{1}:\C[X]\to \R_{+}$ is a power-multiplicative seminorm.
\end{proof}

\begin{lemma}
\label{not-upper-bounded-monomials}
A power-multiplicative seminorm on $\C[X]$ whose restriction to $\C$
is equivalent to $|.|_{\C}$ is upper bounded if and only if there exist $a\in \C$
and $n\in \Z$ such that $|X-a|\leq |n|$.
\end{lemma}
\begin{proof}
One of the implications is clear.
If there exist $a\in \C$ and $n_{0}\in \Z$ such that $|X-a|\leq |n_{0}|$,
then for all $x\in \C$, there exists $n_{x}\in \Z$ such that we have
$$|X-x|\leq |2|\max(|X-a|,|x-a|)\leq\max(|2|.|n_{0}|,|2|.|n_{x}|)=\max(|2.n_{0}|,|2.n_{x}|)$$
so that there exists $m_{x}\in \Z$ such that $|X-x|\leq |m_{x}|$.
If $P$ is a polynomial, let $P=u\times \prod_{x_{i}}(X-x_{i})$
be the decomposition of $P$ in prime factors. Then we have
$|P|=|u_{i}|.\prod_{x_{i}}|X-x_{i}|$ and we get an $n_{P}\in \Z$ such that
$|P|\leq |n_{P}|$.
\end{proof}

\section{Analytic spaces}
We now want to define analytic functions on some convenient
subspaces of harmonious spectra. We first have to define
the value ring for analytic functions at a given place.
This will be given by what we call the multiplicative completion.

\subsection{Ball topologies and multiplicative completions}
\begin{definition}
Let $A$ be ring and $|.|:A\to R$ be a seminorm.
An open ball in $A$ for $|.|$ is a subset of the form
$B(x,|a|):=\{z\in A|\,|z-x|<|a|\}$ for $a\in A$
such that $|a|>0$.
The ball neighborhood topology on $A$, denoted
$\tau^{bn}_{|.|}$ if it exists, is the topology for which the non-empty open balls
form fundamental systems of neighborhoods of the points in $A$.
\end{definition}

We now give the coarsest conditions on a seminorm
$|.|:A\to R$ under which its ball neighborhood topology is well defined and induces
a topological ring structure on $A$.

\begin{definition}
Let $A$ be a ring and $|.|:A\to R$ be a seminorm.
We say that $|.|$ has tiny balls if
\begin{enumerate}
\item for all $|a|>0$, there exists $|a'|>0$ such that
$$B(0,|a'|)+B(0,|a'|)\subset B(0,|a|);$$
\item for all $|a|>0$ and $x\in A$, there exists $|c|>0$
such that
$$x.B(0,|c|)\subset B(0,|a|);$$
\item for all $|a|>0$, there exists $|a'|>0$ such that
$$B(0,|a'|).B(0,|a'|)\subset B(0,|a|).$$
\item for all $|a|>0$, there exists $|a'|>0$ such that
$-B(0,|a'|)\subset B(0,|a|)$.
\end{enumerate}
\end{definition}

\begin{remark}
\label{simpler-condition-tiny-3}
Since $R$ is totally ordered, the third condition is automatic.
Indeed, if $|a|>1$, then there exists
$a'=1$ such that $B(0,|1|).B(0,|1|)\subset B(0,|a|)$. Suppose
that $|a|=1$. If there exists no $c$ in $A$ such that $|c|<1$,
then $B(0,|1|)=\{0\}=B(0,|1|).B(0,|1|)$. If there exists
$c\in A$ such that $|c|<1=|a|$, then $|c|^2<|c|$ and
$B(0,|c|).B(0,|c|)\subset B(0,|c|)\subset B(0,|a|)$.
If $|a|<1$ then $|a|^2<|a|$ and $B(0,|a|).B(0,|a|)\subset B(0,|a|)$.
\end{remark}

\begin{remark}
\label{simpler-condition-tiny-4}
If $|.|$ is square-multiplicative, the fourth condition is also
automatic because $|-a|=|a|$ for all $a\in A$ so that
$-B(0,|a|)=B(0,|a|)$.
\end{remark}

\begin{proposition}
Let $A$ be a ring and $|.|:A\to R$ be a seminorm on $A$
that has tiny balls. Then the ball neighborhood topology
for $|.|$ exists and it induces a ring topology on $A$. 
\end{proposition}
\begin{proof}
Consider the set $\Bc:=\{B(0,|a|)|0<|a|\in |A|\}$ of parts of
$A$. This set is a filter basis since
\begin{enumerate}
\item the intersection $B(0,|a|)\cap B(0,|a'|)$ is equal to $B(0,\min(|a|,|a'|))$;
\item since $R$ is positive, it is not empty because $B(0,1)\in \Bc$, and
$\emptyset\notin \Bc$ because $0\in B(0,|a|)$ for all $|a|>0$.
\end{enumerate}
Now by hypothesis, this filter basis is such that
\begin{enumerate}
\item for all $B(0,|a|)\in \Bc$, there exists $B(0,|a'|)\in \Bc$
such that $B(0,|a'|)+B(0,|a'|)\subset B(0,|a|)$;
\item for all $B(0,|a|)\in \Bc$, there exists $B(0,|a'|)\in \Bc$ such
that $-B(0,|a'|)\subset B(0,|a|)\in \Bc$.
\end{enumerate}
Then by Bourbaki \cite{Boutopo}, Chap. 3, \S 1.2, Proposition 1,
there exists only one topology on $A$ compatible with its addition and such that
$\Bc$ is the filter basis of the filter of neighborhoods of the unit $0$ of
$(A,+)$. This topology is the ball neighborhood topology.
Remark now that by hypothesis,
\begin{enumerate}
\item for all $x\in A$ and $B(0,|a|)\in \Bc$, there exists
$B(0,|c|)\in \Bc$ such that $x.B(0,|c|)\subset B(0,|a|)$.
\item for all $B(0,|a|)\in \Bc$, there exists $B(0,|a'|)\in \Bc$
such that $B(0,|a'|).B(0,|a'|)\subset B(0,|a|)$.
\end{enumerate}
Then by Bourbaki \cite{Boutopo}, Chap. 3, \S 6.3, the topological structure
defined above is a topological ring structure on $A$.
\end{proof}

\begin{definition}
Let $|.|:A\to R$ be a seminorm on a ring $A$. If $|.|$ has tiny
balls, the completion of the separated quotient $A/\overline{\{0\}}$
of $A$ for its topological ring structure is called the separated
completion of $A$ for $|.|$ and denoted $\hat{A}_{|.|}$.
Let $A_{M}$ be the localization of $A$ with respect to the
multiplicative subset $M=M_{|.|}$ of $|.|$. If $|.|:A_{M}\to R$
has tiny balls, we say that $|.|$ has multiplicatively tiny balls and
call the corresponding completion $\Ac(|.|)$ of $A_{M}$
the multiplicative completion of $|.|$.
\end{definition}

\begin{proposition}
\label{multiplicative-tiny}
Let $|.|:K\to R$ be a multiplicative seminorm on a field $K$.
If
\begin{itemize}
\item either there exists $u\in K$ such that $|u|>2$,
\item or there exists $|a|>0$ such that $|b|<|a|$ implies
$|b|=0$,
\end{itemize}
then $K$ has tiny balls for $|.|$ so that $|.|$ induces a ring topology
on $K$. This topology is separated.
\end{proposition}
\begin{proof}
If there exists $|a|>0$ such that $|b|<|a|$ implies $|b|=0$, then
$B(0,|a|)=\{0\}$ (because the kernel of $|.|$ is an ideal that must me
reduced to $\{0\}$ since $K$ is a field) and all conditions for $K$ to have tiny balls
are clearly fulfilled. Moreover, the corresponding topology is the
discrete topology on $K$ and is separated.
Suppose now that for all $|a|>0$, there
exists $|a'|$ such that $0<|a'|<|a|$.
Let $b,c\in K$ be such that $|b|>0$ and $|c|>0$.
Applying the hypothesis to $|b/a|=|b|/|a|$, we show
that there exists $d\in K$ such that
$|d|<|b/a|$, i.e. $0<|d|.|a|<|b|$. This implies the
second condition for $|.|$ to have tiny balls.
Now remark that by hypothesis, there exists $u\in K$ such
that $|u|>2$. Let $a'\in K$ be such that $0<|a'|<|a|$.
Then if we let $d=a/u$, we have $|d|+|d|\leq \frac{|a'|}{2}+\frac{|a'|}{2}=|a'|<|a|$. This
is the second condition for $|.|$ to have tiny balls.
By Remarks \ref{simpler-condition-tiny-3} and \ref{simpler-condition-tiny-4},
since $|.|$ is multiplicative,
we have proved that $|.|$ has tiny balls and this implies that
$|.|$ induces a ring topology on $A$. Since $|.|$ is multiplicative,
its kernel is trivial. Suppose that $x$ and $y$ are two distinct
elements of $K$. Then $|x-y|>0$ in $R$ and we know by what
we did above that there exists $|d|>0$ such that $0<|d|+|d|<|x-y|$.
Now if $x\in B(x,|d|)\cap B(y,|d|)$, then
$$|x-y|\leq |x-z|+|y-z|\leq |u|+|u|<|x-y|,$$
which is a contradiction, so that $B(x,|d|)\cap B(y,|d|)=\emptyset$
and the topology is separated.
\end{proof}

\begin{corollary}
Let $|.|:K\to R$ be a tempered multiplicative seminorm on a field $K$.
Then $K$ has tiny balls for $|.|$.
\end{corollary}
\begin{proof}
From Proposition \ref{multiplicative-tiny}, we are reduced to suppose that
there exists no $|a|>0$ such that $|b|<|a|$ implies $|b|=0$.
We will show that there exists $u\in K$ such that $|u|>2$.
Indeed, if $|x|\leq 2$ for all $x\in K$,
then $|x|^n\leq 2$ for all $n\in \N$ and $x\in K$
but since $R$ has tempered growth, this implies $|x|\leq 1$ for
all $x\in K$. Since there exists $x$ such that $0<|x|<1$, we have
$|1/x|=1/|x|>1$, which gives a contradiction. So there exists a
$u\in K$ such that $|u|>2$. Proposition \ref{multiplicative-tiny}
concludes the proof.
\end{proof}

The definition of the notion of multiplicative completion was given
to get the following.
\begin{corollary}
Let $|.|:A\to R$ be a tempered multiplicative seminorm on a ring $A$.
Then $|.|$ has multiplicatively tiny balls so that the multiplicative completion of $A$
for $|.|$ is well defined and it is isomorphic to the completion of the residue field
$\Frac(A/\pfk_{|.|})$ with respect to its induced seminorm.
\end{corollary}
\begin{proof}
Since $|.|$ is multiplicative, its set of multiplicative elements
is $M_{|.|}=A-\pfk_{|.|}$ where $\pfk_{|.|}$ is the prime ideal that
is the support of $|.|$. The corresponding localization is the local
ring $A_{(\pfk)}$ and $|.|$ factorizes through it. By Proposition
\ref{multiplicative-tiny}, the field $A_{(\pfk)}/\pfk=\Frac(A/\pfk)$ is
separated for the topology induced by $|.|$ and
it is equal to the separated quotient of $A_{(\pfk)}$.
We also know that the extension $|.|:\Frac(A/\pfk)\to R$ has
tiny balls so that its completion is well defined.
It is equal to the multiplicative completion of
$A$ for $|.|$.
\end{proof}

\subsection{The functoriality issue for multiplicative completions}
\label{functoriality-issue}
Let $f:B\to A$ be a ring morphism and $|.|:A\to R$ be a seminorm.
We will now give conditions that ensure that $f$ induces
a morphism
$$f:\Ac(|f(.)|)\to \Ac(|.|)$$
between the multiplicative completions of $|f(.)|:B\to R$
and $|.|:A\to R$.

Remark that the set of multiplicative elements
(see Definition \ref{definition-multiplicative})
for a seminorm is not functorial in ring morphisms meaning that if
$f:B\to A$ is a ring morphism and $|.|:A\to R$ is a seminorm,
then we don't necessarily have
$f(M_{|f(.)|})\subset M_{|.|}$.
This problem does not appear in the case of multiplicative seminorms
but we have to solve it in the non-multiplicative case. This motivates
the following definition.

\begin{definition}
\label{def-functorial-multiplicative}
Let $|.|:A\to R$ be a seminorm on a ring $A$ and $M_{A}\subset A$
be the set of multiplicative elements for $|.|$. We say
that $|.|$ has a functorial multiplicative subset if for all
ring inclusion $f:B\subset A$, we have an inclusion
of $M_{B}\subset M_{A}$ of the corresponding sets of
multiplicative elements for $|.|$.
\end{definition}

\begin{example}
A multiplicative seminorm $|.|:A\to R$ has clearly a functorial multiplicative subset
given by the complement $M_{|.|}=\{a\in A|\,|a|\neq 0\}$ of the corresponding
prime ideal $\pfk_{|.|}=\{a\in A|\,|a|=0\}$.
\end{example}

\begin{remark}
Completions of rings with respect to seminorms
are not functorial in ring homomorphisms.
For example, if $|.|_{2,0}:\Q(T)\to R=\R_{+,trop}\lextimes\R_{+}$
if given by
$$|\sum_{i=i_{0}}^na_{i}T^i|_{2,0}=(e^{-i_{0}},|a_{i_{0}}|_{p})$$
for $a_{i_{0}}\neq 0$, then the natural morphism
$$\Q\to \Q(T)\to R$$
does not induce a topological ring homomorphism
from $\Q_{p}$ to the completion of $\Q(T)$. Indeed,
the inverse image of $B(0,|T|)$ in $\Q$ by the inclusion
$\Q\to \Q(T)$ is $\{0\}$ and it is not open in $\Q$
for the $p$-adic topology. In fact, the completion of $\Q(T)$
for $|.|_{2,0}$ only depends on the $T$-adic valuation
$|P|_{0}=e^{-i_{0}}$ and it gives $\Q((T))$. We thus need a
more general notion of ``completion'' that is functorial
with respect to sequences
$$A\overset{f}{\to} B\overset{|.|}{\to} R$$
with $f:A\to B$ a ring homomorphism and $|.|:B\to R$ a reasonable
seminorm on $B$.
\end{remark}

\begin{remark}
\label{local-numbers}
Let's translate the above remark in terms of numbers.
The notion of local number (or element of a completion) is not functorial
in ring morphisms.
This could be repaired by replacing the usual completion
of $(A,|.|)$ by (say) the subset
$$\A^{1,h}_{(A,|.|)}=\{x\in \A^1_{A}||.(x)|_{|A}\sim |.|\}$$
of the harmonious affine line, which is clearly functorial
by definition. In this sense, an $(A,|.|)$-completed
number would just be a particular seminorm on $A[X]$.
This is not that strange because any number $a$ in $\R$ or $\Q_{p}$
can be seen as such a thing in
$\A^{1,h}_{(\Z,|.|)}$ given by $|.(a)|_{\infty}$ and $|.(a)|_{p}$
respectively.
However, this setting would prevent us from adding numbers
because there is no way to properly add seminorms with values in different
halos. Perhaps one could be inspired by Conway's theory
of surreal numbers \cite{Conway} to deal with this problem.
The functoriality issue of analytic functions is certainly at the heart of our difficulties,
and this is not to be hidden. The idea of this remark will be explored further
in Section \ref{analytic-functions-affine-line}.
\end{remark}

\begin{definition}
\label{functorial-multiplicative-completion}
Let $A$ be a ring and $|.|:A\to R$ be a seminorm on $R$.
We say that $|.|$ has a functorial multiplicative completion if
for all ring homomorphism $f:B\to A$, if we denote
$M_{B}\subset B$ and $M_{A}\subset A$ the corresponding multiplicative subsets,
we have that
\begin{enumerate}
\item $f(M_{B})\subset M_{A}$ (i.e. $|.|$ has a functorial multiplicative subset),
\item $|f(.)|:B\to R$ has multiplicatively tiny balls and $|f(.)|$ extends to
the completion of $B_{M_B}$,
\item for all $|a|>0$ in $|A_{M_{A}}|$, there exists $b$ in $B_{M_{B}}$ such
that $|b|>0$ and $B(0,|b|)\subset f^{-1}(B(0,|a|))$.
\end{enumerate}
If $|.|:A\to R$ is multiplicative and has functorial multiplicative completion,
we will say that it has functorial residue fields.
\end{definition}

\begin{example}
A tempered multiplicative seminorm $|.|:A\to R$ has functorial residue fields
if and only if for all ring homomorphism $f:B\to A$,
the corresponding (injective) morphism
$$f:\Frac(B/\pfk_{B})=K_{B}(|.|)\to K_{A}(|.|)=\Frac(A/\pfk_{A})$$
between the residue fields is continuous, i.e., for all $a\in K_{A}^*$,
there exists $b\in K_{B}^*$ such that $B(0,|b|)\subset f^{-1}(B(0,|a|))$.
\end{example}

\begin{lemma}
A multiplicative seminorm $|.|:A\to \R_{+}$
has functorial residue fields.
\end{lemma}
\begin{proof}
We can reduce to the case of a field $A$.
Let $f:B\to A$ be a field morphism. Remark that given $a\in A^*$,
the condition that there exists $b\in B^*$ such that
$B(0,|b|)\subset f^{-1}(B(0,|a|)$ is clearly fulfilled if $|a|\geq 1$.
Indeed, we then have $B(0,|1|)\subset f^{-1}(B(0,|1|)\subset
f^{-1}(B(0,|a|))$. We thus restrict to the case $|a|<1$. 
If there exists $b\in B$ such that $0<|b|<1$,
then there exists $n\in \N$ such that $|b^n|=|b|^n<|a|$
this implies that $B(0,|b^n|)\subset f^{-1}(B(0,|a|))$.
Otherwise, we have $|B|=\{0,1\}$ and $B(0,|1|)=\{0\}\subset f^{-1}(B(0,|a|))$.
\end{proof}

\subsection{The analytic spectrum of a ring}
\begin{definition}
Let $A$ be a ring. The analytic spectrum of $A$
is the subset $\Speh^a(A)$ of $\Speh^{m}(A)$
of the tempered multiplicative spectrum
given by seminorms $|.(x)|$ that have a functorial
multiplicative completion.
\end{definition}

\begin{definition}
\label{analytic-functions}
Let $A$ be a ring and
$$
R\left(\frac{a_{1},\dots,a_{n}}{b}\right)=\{x\in \Speh^a(A)|\,
|a_{i}(x)|<|b(x)|,\,d|b\Rightarrow d\textrm{ is multiplicative for }|.(x)|\}
$$
be an open rational subset of $\Speh^a(A)$.
Let $\tilde{R}\subset \Speh^{pasm}(A)$ be the subset composed
of $R$ and of the set of $\R_{+}$-valued square-multiplicative seminorms
that have functorial multiplicative completion and are $\Z$-multiplicative.
The ring of analytic series
on $\tilde{R}$ is the completion $\Oc(\tilde{R})$ of $A[b^{-1}]$ for the topology induced
by the natural map
$$A[b^{-1}]\to \prod_{x\in \tilde{R}}\Ac(x).$$
Let $U$ be an open subset of $\Speh^a(A)$.
A rule
$$f:U\to \coprod_{x\in \Speh^a(A)}\Ac(x)$$
such that $f(x)\in \Ac(x)$ is called an analytic function
if for all $x_{0}\in \Speh^a(A)$, there exists an open rational domain
$R=R\left(\frac{a_{1},\dots,a_{n}}{b}\right)$ contained in $U$ and containing $x_{0}$ such that $f_{|R}$ is in the image of $\Oc(\tilde{R})$
by the natural map
$$\Oc(\tilde{R})\to \prod_{x\in R}\Ac(x).$$ 
\end{definition}

\begin{remark}
Another approach to the definition of analytic function (that would prevent
us from using the halo of positive real numbers $\R_{+}$ and follows the viewpoint of
Remark \ref{local-numbers}) will be studied in Section \ref{analytic-functions-affine-line}.
\end{remark}

\subsection{The analytic spectrum of $\Z$}
\begin{proposition}
The analytic and harmonious spectrum of $\Z$ are
identified.
Germs of analytic functions on $\Speh^a(\Z)$ can be
described by figure \ref{figure-germes-fonctions-speha-Z}
(see also figure \ref{figure-berko-Z}).
\begin{figure}[h]
\centering
\includegraphics[width=0.4\textwidth]{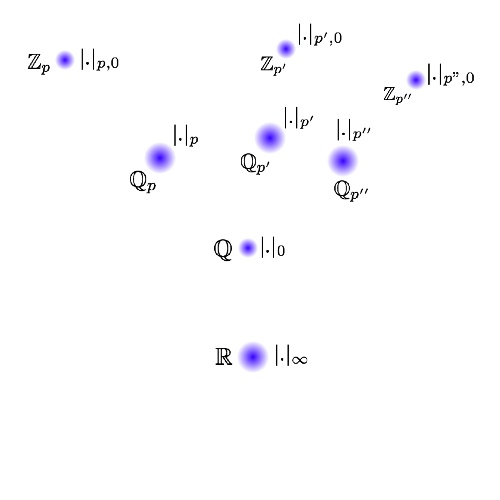}
\caption{Germs of analytic functions on $\Speh^a(\Z)$.}
\label{figure-germes-fonctions-speha-Z}
\end{figure}

\noindent More precisely, depending on the point $x\in \Speh^a(\Z)$,
the corresponding residue field (resp. germs of analytic functions)
are
\begin{itemize}
\item $\F_{p}$ (resp. $\Z_{p}$) if $x=|.|_{p,0}$,
\item $\Q_{p}$ (resp. $\Q_{p}$) if $x=|.|_{p}$,
\item $\R$ (resp. $\R$) if $x=|.|_{\infty}$,
\item $\Q$ (resp. $\Q$) if $x=|.|_{0}$.
\end{itemize}
Moreover, for $m\neq 0$, sections on $\{x|\,0<|m(x)|\}$ are identified
with $\Z[1/m]$ with its trivial topology. In particular, global sections are
identified with $\Z$.

\end{proposition}
\begin{proof}
The statement on residue fields follows from the definition.
If a rational domain $D=\{|n_{1}|<|m|,\dots,|n_{k}|<|m|\}$ contains
the trivial norm, then $n_{i}=0$ for all $i=1,\dots,k$ and $\Oc^h(D)$ is the completion
of $\Z[1/m]$ with respect to the topology induced by the natural map
$$\Z[1/m]\to \Q\times \R\times (\prod_{n}\Q_{n})\times (\prod_{p\nmid m}\F_{p})$$
because the points of $D$ are given by
$D=\{|.|_{0}\}\cup\{|.|_{p},\forall p\}\cup \{|.|_{p,0},\forall p\nmid m\}\cup \{|.|_{\infty}\}$
and the $\R_{+}$-power-multiplicative seminorms are of the form $|.|_{n}$ for
$n\in \Z-\{0,\pm 1\}$ and $|.|_{0,n}$ for $n\nmid m$.
Since $\Q$ above is equipped with the discrete topology (induced by the trivial
valuation $|.|_{0}$), we have $\tilde{\Oc}(D)=\Z[1/m]$ equipped with the discrete topology.
More precisely, germs of analytic functions at $|.|_{0}$ are equal to
$\limind_{m}\Z[1/m]=\Q$.
Now suppose that $D$ does not contain the trivial norm. Then it is
a disjoint union of rational domains of the form
$\{|p|<1\}$ and $\{0<|p|,\,|p|<1\}$, so that we can suppose that
$D$ is of one of these two forms. In the first case, we have $D=\{|.|_{p},|.|_{p,0}\}$,
and $\tilde{\Oc}(D)$ is the completion of $\Z$ with respect
to the natural map
$$\Z\to \Q_{p}\times \F_{p},$$
i.e, $\tilde{\Oc}(D)=\Z_{p}$. More precisely, in this case, $D$
is the smallest rational subset that contains $|.|_{p,0}$ so
that germs of analytic functions around $|.|_{p,0}$ are also
given by $\Z_{p}$.
In the second case, we have $D=\{|.|_{p}\}$ and
$\tilde{\Oc}(D)$ is the completion of $\Z[1/p]$ with respect
to the natural map
$$\Z[1/p]\to \Q_{p}.$$
There is a natural map from the completion $\Z_{p}$ of $\Z$
with respect to $\Z\to \Q_{p}$ to $\tilde{\Oc}(D)$ and $p$ is
invertible in $\tilde{\Oc}(D)$ so that the natural map
$\Z_{p}\to \tilde{\Oc}(D)$ factorizes through $\Q_{p}=\Z_{p}[1/p]$,
so that $\tilde{\Oc}(D)\cong \Q_{p}$ and these are also the germs
of analytic functions at $|.|_{p}$.
If $D=\{1<|2|\}=\{|.|_{\infty}\}$,
$\tilde{\Oc}(D)$ is the completion of $\Z[1/2]$ with respect to
the natural map
$$\Z[1/2]\to \R,$$
so that $\tilde{\Oc}(D)=\R$ and these are also the germs
of analytic functions at $|.|_{\infty}$.
\end{proof}

\subsection{The analytic affine line over $\Z$}
We first describe the set of points of the analytic line over $\Z$.

\begin{proposition}
\label{analytic-affine-line-points}
The points of $\Speh^{m}(\Z[T])$ that are in the analytic affine line
$\Speh^{a}(\Z[T])$ are given by the following list:
\begin{enumerate}
\item The trivial seminorm.
\item $\F_{p}$-points: (non-archimedean) seminorms whose restriction to $\Z$ have
a non-trivial kernel $(p)$. These are described in Lemma \ref{Fp-points}.
\item $\Q_{p}$-points: (non-archimedean) seminorms whose restriction
to $\Z$ give the $p$-adic norm. These are described in Proposition
\ref{Huber-Knebusch} and the only ones that are non-analytic are
those associated in 4 of loc. cit. to a major subset $M\subset |\Q^*|$
with $M=\emptyset$ or $M=|\Q^*|$ (i.e. infinitesimal neighborhoods
of $p$-adic $\bar{\Q}$-points and infinitesimal neighborhood of infinity).
More precisely, they are in bijection with the points of Huber's adic affine line
$\A^{1,ad}_{\Q_{p}}=
\A^1_{\Q_{p}}\underset{\Spec(\Q_{p})}{\times}\Spa(\Q_{p},\Z_{p})$.
\item $\R$-points: (archimedean) seminorms whose restriction to $\Z$
give the usual archimedean norm. These are described in Theorem
\ref{archimedean-harmonious-affine-line} and the only ones that
are non-analytic are those of the form $|.|_{2,a}$ for $a\in \bar{\Q}^*$
or $|.|_{2,\infty}$ in the notations of loc. cit. (i.e. infinitesimal neighborhoods
of archimedean $\bar{\Q}$-points and infinitesimal neighborhood of infinity).
To be precise, archimedean seminorms are $\R_{+}$-valued
seminorms.
\end{enumerate}
\end{proposition}

We now describe the relation of our structural sheaf with
the usual sheaf of holomorphic functions on the complex plane.
\begin{proposition}
Let $\Speh^{arch}(\Z[T])$ be the archimedean open subset of $\Speh^a(\Z[T])$,
defined by the condition $|2(x)|>|1|$. If $R\subset \Speh^{arch}(\Z[T])$
is a rational domain and $R'\subset \C$ is the corresponding rational
domain in the complex plane, the natural inclusion $R'\to R$
induces an isomorphism
$$\Oc(R)\overset{\sim}{\to}\Hc ol(R')$$
where $\Hol(R')$ denotes usual holomorphic functions on $R'$
and $\Oc(R)$ denotes the sections of the structural analytic
sheaf of $\Speh^a(\Z[T])$ on the open subset $R$.
\end{proposition}
\begin{proof}
Remark first that $\Speh^{arch}(\Z[T])$ is identified with $\C/c$ where
$c$ denotes complex conjugation.
Let $\tilde{R}$ be the subset of $\Speh^{pasm}(\Z[T])$ associated
to $R$ as in Definition \ref{analytic-functions}. This set is identified
with the set of spectrally convex compact subsets of $\C/c$.
The map $R'\subset \tilde{R}$
is the inclusion of points in the set of compact subsets.
A function in $\Oc(\tilde{R})$ is locally a limit of a Cauchy sequence
of rational functions with no poles for the topology of convergence
on all compacts. Such a limit is holomorphic and is uniquely
determined by its values on $R'\subset R$. Moreover, every
holomorphic function on $R'$ is a local uniform limit on
all compacts of polynomials. This shows that
$\Oc(R)\overset{\sim}{\to}\Hc ol(R')$ is an isomorphism.
\end{proof}

There is also a nice relation between our structural sheaf and
the sheaf of analytic functions on Berkovich's analytic space.
\begin{proposition}
Let $\Speh^{p}(\Z[T])$ be the $p$-adic open subset of $\Speh^a(\Z[T])$,
defined by the conditions $0<|p(x)|<|1(x)|$ and $\A^{1,ber}_{\Q_{p}}$
be Berkovich's affine line over $\Q_{p}$. The natural map
$f:\A^{1,ber}_{\Q_{p}}\to \Speh^p(\Z[T])$
induces an isomorphism of sheaves
$f^{*}\Oc\to \Oc^{ber}$ between the inverse image of
the sheaf of analytic functions and Berkovich's sheaf of
analytic functions.
\end{proposition}
\begin{proof}
By definition of an analytic point, every point of $\Speh^p(\Z[T])$ is in $\Speh^v(\Z[T])$
and extends to a point in the valuation spectrum
$\Speh^v(\Q_{p}[T])$ of $\Q_{p}[T]$. We thus have a natural
map
$$\Speh^p(\Z[T])\to \Speh^v(\Q_{p}[T])$$
and the natural map
$\A^{1,ber}_{\Q_{p}}\to \Speh^v(\Q_{p}[T])$
factors through its image. The points in $\Speh^p(\Z[T])$
that are not in the image of this map are associated to a major
subset $M\subset |\Q^*|$ as in Proposition \ref{Huber-Knebusch}
and are of the form
$$|\sum_{i}a_{i}(X-a)^i|(x)=\max(|a_{i}|q^i)$$
with $q<|\gamma|$ for all $|\gamma|\in M$ or
$|\gamma|<q$ for all $|\gamma|\in |\Q^*|-M$,
and are generization of the corresponding Gauss point
$$|\sum_{i}a_{i}(X-a)^i|(x)=\max(|a_{i}||\gamma_{M}|^i)$$
for a given $\gamma_{M}\in\Q^*$ associated to $M$. This means
that the corresponding topology on $\Q_{p}(T)$ is coarser
than the Gauss norm topology. The topology on rational functions
on a given open rational domain $R$ of $\Spev^p(\Z[T])$ thus only
depends on the compact subsets of the corresponding rational domain $R'$ of
$\A^{1,ber}_{\Q_{p}}$. Making $R'$ smaller, one can identify this topology
with the topology of uniform convergence on $R'$.
\end{proof}

\section{Another approach to analytic functions}
\label{analytic-functions-affine-line}
The functoriality issue for residue fields alluded to in Section
\ref{functoriality-issue} is quite problematic. This is mainly due to the
fact that the notion of completion we use is not really
adapted to higher rank valuations. Our desire to use non-$\R_{+}$-valued
valuations in analytic geometry is mainly due to the fact that
these cannot be studied by model theoretic (i.e. algebraic, in some sense)
methods, because the archimedean property of $\R_{+}$ is not a first order
logic property. The other reason is the natural appearance of higher rank
valuations in the study of the $G$-topology on non-archimedean analytic spaces.
 We will now look for a different kind of completion, that is easily
seen to be functorial, but have other drawbacks.

\subsection{Definition of functorial generalized completions}
Let $(A,|.|)$ be a seminormed ring.
If $(A,|.|)=(\Z,|.|_{\infty})$ we would like our definition
of completion to give back $\R$. We can start by saying
that the completion is simply the Berkovich affine line
$\A^{1,ber}_{(A,|.|)}$, i.e., the set of multiplicative seminorms
$|.|:A[X]\to \R_{+}$ whose restriction to $A$ is the given
seminorm. In the above case, we get $\C/c$ where
$c$ denotes complex conjugation.

\begin{definition}
The pseudo-archimedean square multiplicative affine line on a seminormed
ring $(A,|.|)$ is the space
$$\A^{1,pasm}(A,|.|):=\Speh^{pasm}_{|.|}(A[X])$$
of pre-archimedean square-multiplicative seminorms
$|.|'$ on $A[X]$ such that $|.|'_{|A}$ is multiplicatively equivalent to $|.|$.
\end{definition}

From now on, we allow ourselves to use various topologies (typically given by
rational subsets with closed or open inequalities) on the spectra $\Speh^{pasm}$.
It seems that the closed inequalities topology is better adapted to non-archimedean
spaces.

Recall that there is a natural map
$$A\to \A^{1,pasm}(A,|.|)$$
given by $a\mapsto [P\mapsto |P(a)|]$.
This induces a topology on $A$, and we can think
of the topological space $\A^{1,pasm}(A,|.|)$ as a kind
of generalized completion of $A$ for this topology.
Moreover, if $f:(A,|.|_{A})\to (B,|.|_{B})$ is a ring morphism
such that $|.|_{B}\circ f:A\to R$ is (multiplicatively) equivalent to $|.|_{A}$, then
$f$ extends to a continuous map
$$\hat{f}:\A^{1,pasm}(A,|.|_{A})\to \A^{1,pasm}(B,|.|_{B}).$$
Let $\SNRings$ denote the category of seminormed rings with morphisms
respecting the multiplicative equivalence class of seminorms.
The functor
$$\A^{1,pasm}:\SNRings\to \Top$$
can be seen as a functorial generalized completion
functor. For more flexibility, we will extend this a bit.

\begin{definition}
Let $\Speh^\bullet$ be a sub-quotient functor of $\Speh^{pasm}:\Rings\to \Top$.
Let $\A^{1,\bullet}:\SNRings\to \Top$ be the corresponding version of the affine line
given by
$$\A^{1,\bullet}(A,|.|):=\{x\in \Speh^\bullet(A[X])|\,|.(x)|_{|A}\sim |.|\}.$$
Suppose further that the image of $A$ in $\A^{1,pasm}(A,|.|)$ goes functorially in
$\A^{1,\bullet}(A,|.|)$.
The $\bullet$-completion of $(A,|.|)$ is the space
$\A^{1,\bullet}(A,|.|)$.
\end{definition}

Let us consider the functor $\Speh^\bullet=\Speh^m\subset \Speh^{pasm}$.
The tempered multiplicative affine line $\A^{1,m}(\Z,|.|_{\infty})$
contains all seminorms of the form $P\mapsto |P(x)|_{\infty}$ for $x\in \R$.
It also contains finer infinitesimal points. Moreover, the tempered multiplicative
affine line $\A^{1,m}(\Z,|.|_{p})$ is equal to the valuative affine line
$\A^{1,v}(\Q,|.|_{p})$ whose separated quotient is the Berkovich
affine line $\Mc(\Q[T],|.|_{p})$ of multiplicative seminorms $|.|:\Q[T]\to \R_{+}$
such that $|.|_{|\Q}\sim|.|_{p}$. It contains $\Q_{p}$ as a subset but
is much bigger.

\begin{remark}
The main advantage of such a completion procedure is that it is functorial.
Its main drawback is that it does not give rings but just topological spaces
of local functions.
\end{remark}

\subsection{Definition of foanalytic functions}
Let $\Speh^\bullet$ be a sub-quotient functor of $\Speh^{pasm}$
as in last Section.
Let $A$ be a ring and
$$
U=R\left(\frac{a_{1},\dots,a_{n}}{b}\right)=\{x\in \Speh^\bullet(A)|\,
|a_{i}(x)|\leq |b(x)|\neq 0,\,d|b\Rightarrow d\textrm{ is multiplicative for }|.(x)|\}
$$
be a (closed-inequalities) rational domain in $\Speh^\bullet(A)$.
Let $\A^{1,\bullet}_{U}\subset \Speh^{\bullet}(A[X])$ be the set
of seminorms $|.|:A[X]\to R$ such that $|.|_{|A}$ is multiplicatively
equivalent to an $x\in U$.

\begin{definition}
The space of foanalytic \footnote{Foanalytic is a shortcut for functorial analytic.} functions on $U\subset \Speh^\bullet(A)$
is the space $\Bc^{\bullet}(U)$ of continuous sections of
the natural projection
$$\pi:\A^{1,\bullet}_{U}\to U.$$ 
\end{definition}

The evaluation morphisms for polynomials at elements $a\in A$ induces
a natural map
$$\mathrm{ev}:A\to \Bc^\bullet(U)$$
which factors in
$$A[1/b]\to \Bc^\bullet(U)$$
because $b$ is multiplicative for every seminorm $x$ in $U$.
This shows that the Zariski pre-sheaf $\Oc_{alg}$ on rational
domains naturally maps to the pre-sheaf $\Bc^\bullet$.

The possibility of varying the sub-quotient functor $\Speh^\bullet$,
and eventually taking subfunctors of $\Bc^\bullet$,
makes our theory quite flexible, but the comparison with
usual global analytic spaces does not seem to be an easy task.
We have however included these constructions in our work
because they are the only way we found of having a functorial
version of the completion procedure that seems necessary to define
a flexible notion of analytic function.

\section*{Acknowledgments}
The author thanks P. Almeida, V. Berkovich, M. Berm\'udez, A. Ducros,
I. Fesenko, R. Huber, J. Poineau, A. Thuillier, B. Toen, G. Skandalis and M. Vaqui\'e for useful
discussions around the subject of global analytic geometry. He is particularly grateful
to R. Huber who authorized him to reproduce his own ideas in subsection \ref{huber}.
He also thanks University of Paris 6 and Jussieu's Mathematical Institute for financial
support and excellent working conditions.

\bibliographystyle{alpha}
\bibliography{/Users/fpaugam/Documents/travail/fred}

\end{document}